\documentclass[11pt, reqno, twoside, letterpaper]{amsart}

\usepackage[
 letterpaper, twoside,
 inner=1.35in, outer=1.40in,
 top=1.25in,  bottom=1.25in,
 headsep=16pt, footskip=30pt,
]{geometry}

\usepackage{amsmath}
\usepackage{mathtools}

\usepackage{libertinus}
\usepackage{microtype}

\makeatletter
\g@addto@macro\normalsize{%
 \setlength\abovedisplayskip{13pt plus 3pt minus 4pt}%
 \setlength\belowdisplayskip{13pt plus 3pt minus 4pt}%
 \setlength\abovedisplayshortskip{0pt plus 3pt}%
 \setlength\belowdisplayshortskip{9pt plus 3.5pt minus 3pt}%
}
\makeatother

\usepackage{xcolor}
\usepackage{enumitem}
\usepackage{etoolbox}
\usepackage{csquotes}    

\usepackage{cite}

\usepackage{pdflscape}
\usepackage{array}
\usepackage{booktabs}
\usepackage{ragged2e}
\usepackage{tabularx}
\usepackage{graphicx}
\usepackage{eso-pic}

\newcolumntype{L}[1]{>{\RaggedRight\arraybackslash}p{#1}}
\newcolumntype{C}[1]{>{\Centering\arraybackslash}p{#1}}
\newcolumntype{Z}{>{\RaggedRight\arraybackslash}X}

\usepackage{fancyhdr}
\fancypagestyle{landscapetable}{%
  \fancyhf{}
  \fancyfoot[C]{\thepage}
  
}

\numberwithin{equation}{section}

\theoremstyle{plain}
\newtheorem{X}{X}[section]
\newtheorem{theorem}[X]{Theorem}
\newtheorem{lemma}[X]{Lemma}
\newtheorem{corollary}[X]{Corollary}
\newtheorem{proposition}[X]{Proposition}
\newtheorem*{theorem*}{Theorem}

\newtheorem*{conjecture*}{Conjecture}
\newtheorem{hypothesis}[X]{Hypothesis}
\newtheorem*{hypothesis*}{Hypothesis}

\theoremstyle{definition}

\theoremstyle{remark}

\newtheorem*{remark*}{Remark}

\makeatletter
\pretocmd{\@thm}{%
  \setlength{\parindent}{\z@}%
  \setlength{\parskip}{4pt plus 1pt}
}{}{\PackageWarning{preamble}{Could not patch \string\@thm}}
\makeatother

\usepackage{footnotehyper}   

\newcommand{\KeepTogether}[1]{%
  \AddToHook{env/#1/before}{\par\savenotes\vbox\bgroup}%
  \AddToHook{env/#1/after}{\egroup\spewnotes}%
}
\KeepTogether{theorem}
\KeepTogether{theorem*}
\KeepTogether{lemma}
\KeepTogether{corollary}
\KeepTogether{proposition}
\KeepTogether{definition}
\KeepTogether{conjecture}
\KeepTogether{conjecture*}
\KeepTogether{hypothesis}
\KeepTogether{hypothesis*}

\allowdisplaybreaks[1]

\renewcommand{\le}{\leqslant}
\renewcommand{\ge}{\geqslant}

\makeatletter
\apptocmd{\thebibliography}{%
 \raggedright
 \@rightskip=\z@ \@plus 3em
 \rightskip=\@rightskip
 \parfillskip=\z@ \@plus 1fil
 \frenchspacing
}{}{\PackageWarning{preamble}{Could not patch thebibliography}}
\makeatother

\makeatletter
\patchcmd{\@tocline}
 {\hfil}
 {\leaders\hbox{$\m@th\mkern 4.5mu\hbox{.}\mkern 4.5mu$}\hfill}
 {}{\PackageWarning{preamble}{Could not patch \string\@tocline}}
\makeatother

\usepackage{hyperref}
\hypersetup{
 unicode=true,
 pdftitle={Biases in the distribution of primes in short intervals},
 pdfauthor={Tristan Freiberg},
 pdfsubject={Number theory},
 pdfstartview={FitH},
 pdfmenubar=false,
 pdffitwindow=false,
 pdfnewwindow=true,
 bookmarksnumbered=true,
 linktoc=all,
 colorlinks=true,
 linkcolor={black},
 citecolor={black},
 filecolor={black},
 urlcolor={black},
}
\newcommand{\DOI}[1]{\href{https://doi.org/#1}{doi:#1}}
\newcommand{\ARXIV}[1]{\href{https://doi.org/10.48550/arXiv.#1}{arXiv:#1}}

\title[Biases in the distribution of primes in short intervals]{Biases in the distribution of primes in short intervals}
\author[T. Freiberg]{Tristan Freiberg}
\address{Montr\'eal, Canada}
\thanks{The use of AI in developing the proof and preparing the exposition is described in Appendix~\ref{app:ai-provenance}}
\date{\today}

\begin{document}

\begin{abstract}
Assuming a suitably uniform Hardy--Littlewood prime tuples hypothesis, we obtain a second-order asymptotic for the proportion of short intervals containing a prescribed number of primes, when the interval length is comparable to the average prime spacing. The leading term is Poisson, but the arithmetic correction differs from the binomial correction in Cram\'er's independent model and predicts a stronger bias toward counts near the mean. The proof combines inclusion--exclusion with singular-series estimates of Montgomery and Soundararajan and a finite-sieve argument that controls the required alternating sums as the truncation order grows. We also give a refined random model that reproduces the correction and numerical comparisons that support the prediction.
\end{abstract}

\maketitle

\tableofcontents
\clearpage

\section{Introduction}
\label{sec:intro}

Let $\pi(x)$ denote the number of primes less than or equal to $x$. One way to formulate the prime number theorem is as follows: for fixed $\lambda > 0$, 
\begin{equation*}
\frac{1}{N}\sum_{n = 1}^{N} \left(\pi(n + \lambda\log n) - \pi(n)\right) \sim \lambda \qquad (N \to \infty).
\end{equation*}
Thus, on average over $n \le N$, the short interval $(n, n + \lambda\log n]$ contains $\lambda$ primes. 

Cram\'er's random model \cite{CRA1935, CRA1936} predicts a Poisson distribution for the number of primes in short intervals; see \cite[equation~(15)]{GRA1995} and \cite[Exercise~1]{SOU2007}. Specifically, for each fixed real $\lambda > 0$ and integer $m \ge 0$,
\begin{equation*}
\frac{1}{N}\left|\{1 \le n \le N : \pi(n + \lambda\log n) - \pi(n) = m\}\right| \sim \frac{e^{-\lambda}\lambda^{m}}{m!} \qquad (N \to \infty).
\end{equation*}
Gallagher \cite{GAL1976} showed that this follows from a sufficiently uniform version of the Hardy--Littlewood prime $k$-tuples conjecture. Earlier, Hooley \cite[p.~137, Problem~(iii)]{HOO1973} had remarked that a suitable form of those conjectures implies an exponential distribution for the normalized gaps $(p_{n + 1} - p_{n})/\log p_{n}$, where $p_{n}$ denotes the $n$th prime. He described this as a gamma distribution with parameter $1$.

Gallagher's proof used the method of moments, with the key input that the singular series in the Hardy--Littlewood prime $k$-tuples conjecture  \cite{HL1923} is of average order $1$. The author \cite[Theorem~4.1]{FRE2018} gave an alternative proof by direct inclusion-exclusion, also treating the joint distribution of prime counts in several disjoint short intervals. This approach allows error terms in the prime-tuple estimates to be tracked explicitly. It was suggested in \cite[Section~5]{FRE2018} that combining it with the more precise singular-series average of Montgomery and Soundararajan \cite[equation~(17)]{MS2004} could yield lower-order terms in the conditional asymptotic for the distribution of primes in short intervals. We carry out this program here.

The general results require some preparation, so we state here a representative theorem for intervals with left endpoint between $1$ and $N$. The more general statements, with substantially weaker hypotheses, appear in Section~\ref{sec:prime-tuples-to-prime-counts}.

Write $\mathfrak{S}(\mathcal{H})$ for the singular series associated with a finite set $\mathcal{H}$ of integer shifts; its definition is recalled in Subsection~\ref{subsec:hardy-littlewood-residual}. We use the following strong form of the Hardy--Littlewood prime $k$-tuples conjecture.

\begin{hypothesis}[Uniform Hardy--Littlewood]
\label{con:uniform-hardy-littlewood-hypothesis}
For every fixed $U > 0$ and $\epsilon > 0$,
\begin{equation*}
\left|\{1 \le n \le N : n + h \text{ is prime for every } h \in \mathcal{H}\}\right| = \mathfrak{S}(\mathcal{H})\int_{2}^{N}\frac{dt}{(\log t)^{k}} + O_{U,\epsilon}\left(N^{1/2 + \epsilon}\right),
\end{equation*}
uniformly for integers 
\begin{equation*}
2 \le H \le U\log N, \qquad 2 \le k \le \log H,
\end{equation*}
and all $k$-element subsets $\mathcal{H}$ of $\{1,\ldots,H\}$. In particular, the implied constant is independent of $k$ and $\mathcal{H}$.
\end{hypothesis}

\begin{theorem}
\label{thm:intro-corrected-poisson}
Assume Hypothesis \ref{con:uniform-hardy-littlewood-hypothesis}. Fix an integer $m \ge 0$ and real numbers $0 < V \le U$. For integers $H$ satisfying $V\log N \le H \le U\log N$, put
\begin{equation*}
\lambda := \frac{H}{\log N - 1}.
\end{equation*}
Then, as $N \to \infty$, uniformly over this range of $H$,
\begin{multline*}
\frac{1}{N}\left|\{1 \le n \le N : \pi(n + H) - \pi(n) = m\}\right| \\ = \frac{e^{-\lambda}\lambda^{m}}{m!}\left\{1 - \frac{\log H + \log(2\pi) + \gamma - 1}{2H}\left((m - \lambda)^{2} - m\right)\right\} + o\left(\frac{1}{H}\right),
\end{multline*}
where $\gamma$ is the Euler--Mascheroni constant.
\end{theorem}

The hypothesis above is chosen for simplicity. Section~\ref{sec:prime-tuples-to-prime-counts} replaces it by weaker hypotheses on averages of prime-tuple counts and allows the left endpoints to range over $(M, M + N]$.

Related biases occur in the residue classes of consecutive terms of arithmetic sequences. Lemke Oliver and Soundararajan \cite{LS2016} used the Hardy--Littlewood prime $k$-tuples conjecture to predict lower-order terms favoring certain residue patterns among consecutive primes. David, Devin, Nam and Schlitt \cite{DDNS2022} developed an analogous explanation for consecutive sums of two squares, using conjectures for correlations of that sequence and estimates for averages of the associated singular series. Our use of the term ``bias'' is in the same spirit: the leading Poisson law is accompanied by an arithmetic correction favoring certain prime counts over others. Here the statistic is the number of primes in a short interval, rather than the residue pattern of consecutive terms.

\subsection{Beyond the Poisson approximation}
\label{subsec:beyond-poisson}

The correction in Theorem~\ref{thm:intro-corrected-poisson} differs even from the correction obtained by retaining the binomial distribution in the independent model. To see this, consider $H$ independent coin flips, each with success probability $\lambda/H$, so that the expected number of successes is $\lambda$. The probability of exactly $m$ successes is
\begin{equation*}
b_{m} := \binom{H}{m}\left(\frac{\lambda}{H}\right)^{m}\left(1 - \frac{\lambda}{H}\right)^{H - m}.
\end{equation*}
For fixed $m$ and $\lambda$ in a compact subset of $(0,\infty)$, expanding the finite product in the binomial coefficient and the logarithm of the last factor gives
\begin{equation*}
\log\left(\frac{b_{m}}{e^{-\lambda}\lambda^{m}/m!}\right) = -\frac{m(m - 1)}{2H} + \frac{m\lambda}{H} - \frac{\lambda^{2}}{2H} + O(H^{-2}).
\end{equation*}
Thus, writing
\begin{equation*}
C_{2}(m; \lambda) := (m - \lambda)^{2} - m,
\end{equation*}
the second Charlier polynomial in this normalization, we obtain
\begin{equation*}
b_{m} = \frac{e^{-\lambda}\lambda^{m}}{m!}\left\{1 - \frac{C_{2}(m; \lambda)}{2H} + O(H^{-2})\right\}.
\end{equation*}
The theorem has the same polynomial, but replaces its coefficient $1/(2H)$ by
\begin{equation*}
\frac{\log H + \log(2\pi) + \gamma - 1}{2H}.
\end{equation*}
The arithmetic correction is therefore larger by a factor asymptotic to $\log H$, wherever $C_{2}(m; \lambda)$ stays away from zero.

Here $\lambda = H/(\log N - 1)$ matches the averaged density to the precision required. It is not the local mean at every starting point: Cram\'er's model assigns probability $1/\log n$ to the integer $n$. Averaging over this varying density changes the comparison only by $O(H^{-2})$ when $H \asymp \log N$, as the estimates in Section~\ref{sec:density-averaging} explain. It does not account for the correction of order $(\log H)/H$.

The distinction comes from arithmetic dependence. Independent trials assign the same probability to every set of $k$ successes, whereas the Hardy--Littlewood conjecture weights each set of shifts by its singular series. Gallagher's result that these weights have average order $1$ explains why independence gives the correct first-order prediction. The more precise singular-series average retains a lower-order contribution, which inclusion-exclusion turns into the correction above.

The sign also explains the bias: the correction increases the Poisson mass where $(m - \lambda)^{2} < m$ and decreases it where the reverse inequality holds. Thus counts near the mean are favored, with a stronger correction than the independent binomial model predicts.

Unconditionally, even the existence of infinitely many such intervals for every fixed $\lambda > 0$ and $m \ge 0$ required further progress. Building on Maynard's breakthrough on small gaps between primes \cite{MAY2015, MAY2016}, the author \cite[Theorem~1.1]{FRE2016} proved that
\begin{equation*}
\left|\{1 \le n \le N : \pi(n + \lambda\log n) - \pi(n) = m\}\right| \ge N^{1 - o(1)}.
\end{equation*}
This establishes the occurrence of every prescribed prime count on this scale, although it falls short of the positive limiting proportion predicted by the Poisson conjecture. For $\lambda$ sufficiently small in terms of $m$, Mastrostefano improved this lower bound to $\gg_{m,\lambda} N/\log N$ \cite[Theorem~1.2]{MAS2019a}, and subsequently to $\gg_{m,\lambda} N$ \cite[Theorem~1.1]{MAS2019b}. Thus a positive proportion of these intervals contain exactly $m$ primes in this restricted range of parameters. For sufficiently large $m$, Sono \cite[Corollary~1.2]{SON2026} gives an explicit quantitative refinement of the positive-proportion bound, again with $\lambda$ restricted in terms of $m$. Gallagher also obtained an unconditional exponential upper bound for the upper tail: for positive constants $\nu \ge \lambda \ge 1$,
\begin{equation*}
\left|\{n \le N : \pi(n + \lambda\log N) - \pi(n) > \nu\}\right| \le (1 + o(1)) N e^{-c\nu/\lambda},
\end{equation*}
where $c > 0$ is an absolute constant \cite[Theorem~2]{GAL1976}.

\subsection{Organization of the paper}
\label{subsec:organization-of-paper}

Section~\ref{sec:notation} records our notation and conventions. Section~\ref{sec:back-of-the-envelope} sketches the argument with the errors suppressed, showing how inclusion--exclusion produces the Poisson main term and its arithmetic correction. Section~\ref{sec:setup-and-bookkeeping} then defines the quantities and residuals needed to make this calculation precise, and Section~\ref{sec:bonferroni-sandwich} gives the unconditional Bonferroni sandwich with all residuals retained.

The technical heart of the proof is Section~\ref{sec:singular-series-finite-sieve}. A finite sieve allows us to control the alternating singular-series sums as the truncation order grows, using precise estimates only for low moments. This also gives the required bounds for the singular-series and truncation residuals. Section~\ref{sec:density-averaging} uses Taylor expansion about the averaged density to replace the integrated approximation by an expression at a single effective mean. Section~\ref{sec:prime-tuples-to-prime-counts} identifies sufficient Hardy--Littlewood hypotheses, which supply the sole unproved input, and combines the estimates to prove the general conditional theorem and its introductory special case.

Section~\ref{sec:random-model-arithmetic-correction} explains the arithmetic correction through a refinement of Cram\'er's random model. Appendix~\ref{sec:numerical-comparisons} compares the approximations numerically with the Poisson and binomial predictions. Appendix~\ref{app:gauss-to-cramer} gives a historical perspective, beginning with Gauss's prime-count tables, and Appendix~\ref{app:ai-provenance} records the provenance of the work, the use of AI and the verification of the arguments.

\subsection*{Acknowledgments}
\label{subsec:acknowledgments}

I thank Andrew Granville for his encouragement and for a calculation he sent me in June 2023, which clarified the treatment of the logarithmic integral and corrected the coefficient of the secondary term.

\section{Notation}
\label{sec:notation}

The set of primes is denoted by $\mathcal{P}$, its indicator by $\mathbf{1}_{\mathcal{P}}$, and the number of primes less than or equal to $x$ by $\pi(x)$. Sums and products indexed by $p$ run over primes. Euler's totient function is denoted by $\phi$, and we write $\gcd(a, b)$ for the greatest common divisor of integers $a$ and $b$.

For a finite set $S$, we write $|S|$ for its cardinality, and put $[H] := \{1, \ldots, H\}$ for a positive integer $H$. The set $\mathcal{H} \bmod p$ consists of the residue classes modulo $p$ represented in $\mathcal{H}$. Empty sums and products have values $0$ and $1$, respectively. Binomial coefficients with nonnegative integer arguments are taken to be zero when the lower argument is larger than the upper argument.

The Euler--Mascheroni constant is denoted by $\gamma = 0.577215\ldots$.

The interval length $H$, the lower endpoint $M$ of the averaging range, and the length $N$ of that range are integers satisfying $H \ge 1$, $M \ge 0$ and $N \ge 2$. For an integer $n \ge 0$, put
\begin{equation*}
X(n; H) := \pi(n + H) - \pi(n).
\end{equation*}
The proportion of intervals containing exactly $m$ primes is
\begin{equation*}
P(m; H, M, N) := \frac{1}{N}\left|\{M < n \le M + N : X(n; H) = m\}\right|,
\end{equation*}
where $m \ge 0$ is an integer and $n$ ranges over integers. Whether $m$ is fixed or allowed to vary is specified in each result.

The singular series of a finite set $\mathcal{H}$ of integers is
\begin{equation*}
\mathfrak{S}(\mathcal{H}) := \prod_{p}\left(1 - \frac{|\mathcal{H} \bmod p|}{p}\right)\left(1 - \frac{1}{p}\right)^{-|\mathcal{H}|}.
\end{equation*}
The set $\mathcal{H}$ is admissible if $|\mathcal{H} \bmod p| < p$ for every prime $p$, equivalently if $\mathfrak{S}(\mathcal{H}) > 0$, and inadmissible otherwise. In particular, $\mathfrak{S}(\varnothing) = \mathfrak{S}(\{h\}) = 1$ for every integer $h$.

For an integer $k \ge 0$, the tuple average $T(k; H, M, N)$ is the average number of ways to choose $k$ primes from $(n, n + H]$, over the integers $M < n \le M + N$. The corresponding singular-series sum $S(k; H)$ is over the $k$-element subsets of $[H]$, each counted once. Thus
\begin{equation*}
T(k; H, M, N) := \frac{1}{N}\sum_{n \, = \, 1}^{N}\binom{X(M + n; H)}{k}, \qquad S(k; H) := \sum_{\substack{\mathcal{H} \, \subseteq \, [H] \\ |\mathcal{H}| \, = \, k}}\mathfrak{S}(\mathcal{H}).
\end{equation*}
The logarithmic density factor is
\begin{equation*}
J(k; M, N) := \frac{1}{N}\int_{\max\{2, M\}}^{M + N}\frac{dt}{(\log t)^{k}}.
\end{equation*}
We write $G(k; H)$ for the two-term approximation to $S(k; H)$:
\begin{equation*}
G(k; H) := \frac{H^{k}}{k!}\left(1 - \binom{k}{2}\eta(H)\right), \qquad \eta(H) := \frac{\log H + \log(2\pi) + \gamma - 1}{H}.
\end{equation*}

We reserve $\mathbb{P}$ for probability and write $\mathbb{E}Z$ for the expectation of a random variable $Z$. The underlying probability space is specified where it is introduced. The Poisson mass at $m$ with mean $u \ge 0$ is $p_{m}(u) := e^{-u}u^{m}/m!$, with values at $u = 0$ understood by continuity; we put $p_{m}(u) = 0$ for negative integers $m$. The corrected Poisson approximation is
\begin{equation*}
Q(m; u, H) := p_{m}(u)\left\{1 - \frac{\eta(H)}{2}\left((m - u)^{2} - m\right)\right\}.
\end{equation*}
Its average over the logarithmic density is denoted by $F(m; H, M, N)$, and the averaged density parameter by $\mu = \mu(H, M, N)$:
\begin{equation*}
F(m; H, M, N) := \frac{1}{N}\int_{\max\{2, M\}}^{M + N}Q\left(m; \frac{H}{\log t}, H\right)dt, \qquad \mu := HJ(1; M, N).
\end{equation*}
The exact mean prime count is $T(1; H, M, N)$; $\mu$ is its logarithmic-integral approximation. 

We use $\lambda$ for a Poisson parameter, with its precise choice specified in each occurrence. The indexed coefficients $\lambda_{d}$ in the proof of Lemma~\ref{lem:uniform-survivor-bound} are local sieve weights and are unrelated to this parameter.

The signed residuals $\Delta_{\mathrm{HL}}$, $\Delta_{\mathrm{MS}}$, $\Delta_{\mathrm{tr}}$ and $\Delta_{\mathrm{av}}$ record, respectively, the Hardy--Littlewood approximation, the singular-series approximation, truncation of the exponential series, and replacement of $F$ by $Q$ evaluated at $\mu$. Their precise definitions are given in Section~\ref{sec:setup-and-bookkeeping}. The corresponding weighted or integrated residuals are denoted by $R_{\mathrm{HL}}$, $R_{\mathrm{MS}}$ and $R_{\mathrm{tr}}$, and their sum by $\mathcal{R}$. Dependence on $H$, $M$ and $N$ is sometimes suppressed, as indicated in the text.

For functions $f$ and $g$, with $g$ positive, we write $f = O(g)$, or equivalently $f \ll g$, if $|f| \le Cg$ throughout the range under consideration for some constant $C > 0$. The notation $g \gg f$ has the same meaning. For positive functions, $f \asymp g$ means that $f \ll g$ and $g \ll f$. Subscripts indicate permitted dependence of the implied constant. Any dependence on additional fixed parameters is stated in the relevant result. Uniformity in a parameter means that the implied constant is independent of that parameter throughout the stated range.

We write $f = o(g)$ if $f/g \to 0$, and $f \sim g$ if $f/g \to 1$. The limiting parameter, usually $N$ or $H$, and the allowed dependence of the other parameters are specified in the text. A uniform $o(g)$ bound means that the supremum of $|f/g|$ over the stated parameter ranges tends to zero.

\section{Back of the envelope}
\label{sec:back-of-the-envelope}

We first sketch the argument assuming $M = 0$, suppressing some parameters and temporarily ignoring all errors. The notation in this section is provisional and will be superseded by the formal definitions in Section~\ref{sec:setup-and-bookkeeping}. We use $\approx$ to indicate a proposed approximation whose error remains to be considered.

Let $H \ge 1$, $N \ge 2$ and $0 \le m \le H$ be integers. Write $\pi(x)$ for the number of primes less than or equal to $x$, and put
\begin{equation*}
P(r) := \frac{1}{N}\left|\{n \le N : \pi(n + H) - \pi(n) = r\}\right| \qquad (r \ge 0).
\end{equation*}
Here $n$ and $r$ are integers, and vertical bars denote cardinality. We wish to estimate $P(m)$. 

Hardy--Littlewood estimates concern prescribed prime tuples, so our first task is to express $P(m)$ in terms of counts of such tuples. Write $[H] = \{1,\ldots,H\}$ and let $\mathbf{1}_{\mathcal{P}}$ be the indicator of the primes $\mathcal{P}$. For $k \ge 0$, put
\begin{equation*}
T(k) := \frac{1}{N}\sum_{\substack{\mathcal{H} \, \subseteq \, [H] \\ |\mathcal{H}| \, = \, k}}\sum_{n \, = \, 1}^{N}\prod_{h \, \in \, \mathcal{H}}\mathbf{1}_{\mathcal{P}}(n + h).
\end{equation*}
An interval containing $r$ primes contributes once for each choice of $k$ of them. Interchanging the sums and grouping the intervals by their prime counts therefore gives
\begin{equation*}
T(k) = \frac{1}{N}\sum_{n \, = \, 1}^{N}\binom{\pi(n + H) - \pi(n)}{k} = \sum_{r \, = \, 0}^{H}\binom{r}{k}P(r).
\end{equation*}
We use the usual conventions for empty products and for binomial coefficients with nonnegative integer arguments; in particular, $T(0) = 1$ and $T(k) = 0$ for $k > H$.

Taking $k = m$ counts the intervals we want, but also counts intervals with more than $m$ primes:
\begin{equation*}
T(m) = P(m) + (m + 1)P(m + 1) + \sum_{r \, = \, m + 2}^{H}\binom{r}{m}P(r) \ge P(m).
\end{equation*}
To remove the contribution from intervals with exactly $m + 1$ primes, subtract $(m + 1)T(m + 1)$. This over-corrects the contributions from intervals with still more primes:
\begin{equation*}
T(m) - (m + 1)T(m + 1) = P(m) - \sum_{r \, = \, m + 2}^{H}\binom{r}{m}(r - m - 1)P(r) \le P(m).
\end{equation*}
Adding the next correction removes the contribution from intervals with exactly $m + 2$ primes and gives another overcount:
\begin{equation*}
T(m) - (m + 1)T(m + 1) + \binom{m + 2}{m}T(m + 2) \ge P(m).
\end{equation*}
Continuing in this way gives the exact inclusion--exclusion identity
\begin{equation*}
P(m) = \sum_{\ell \, = \, 0}^{H - m}(-1)^{\ell}\binom{m + \ell}{m}T(m + \ell).
\end{equation*}
For estimates, we stop earlier. If
\begin{equation*}
B(m, j) := \sum_{\ell \, = \, 0}^{j}(-1)^{\ell}\binom{m + \ell}{m}T(m + \ell),
\end{equation*}
then the Bonferroni inequalities give
\begin{equation*}
B(m, L + 1) \le P(m) \le B(m, L)
\end{equation*}
for every even integer $L \ge 0$. We shall choose $L$ to grow with $H$, so that both bounds can be approximated by the same expression to the required precision.

The definition of $T(k)$ shows where the Hardy--Littlewood conjecture enters. For each set of shifts $\mathcal{H}$, it predicts
\begin{equation*}
\sum_{n \, = \, 1}^{N}\prod_{h \, \in \, \mathcal{H}}\mathbf{1}_{\mathcal{P}}(n + h) \approx \mathfrak{S}(\mathcal{H})\int_{2}^{N}\frac{dt}{(\log t)^{k}}, \qquad k = |\mathcal{H}|,
\end{equation*}
where the singular series $\mathfrak{S}(\mathcal{H})$ accounts for congruence restrictions on the shifts. Summing these proposed main terms gives
\begin{equation*}
T(k) \approx J(k)S(k), \qquad J(k) := \frac{1}{N}\int_{2}^{N}\frac{dt}{(\log t)^{k}}, \qquad S(k) := \sum_{\substack{\mathcal{H} \, \subseteq \, [H] \\ |\mathcal{H}| \, = \, k}}\mathfrak{S}(\mathcal{H}).
\end{equation*}
The prime-tuple counts have thus led us to an average of the singular series. Montgomery and Soundararajan's estimate \cite[equation~(17)]{MS2004}, for fixed $k$ and after division by $k!$ to pass from ordered tuples to subsets, suggests the replacement
\begin{equation*}
S(k) \approx G(k), \qquad G(k) := \frac{H^{k}}{k!}\left(1 - \binom{k}{2}\eta(H)\right), \qquad \eta(H) := \frac{\log H + \log(2\pi) + \gamma - 1}{H},
\end{equation*}
where $\gamma$ is the Euler--Mascheroni constant. The fixed-$k$ estimate alone does not justify this substitution throughout our growing sums. We will need to control the combined contribution of the singular-series residuals.

Substituting $J(k)G(k)$ for $T(k)$ in $B(m,j)$ and moving the finite sum inside the integral gives
\begin{equation*}
B(m, j) \approx \frac{1}{N}\int_{2}^{N}\frac{(H/\log t)^{m}}{m!}\sum_{\ell \, = \, 0}^{j}\frac{(-H/\log t)^{\ell}}{\ell!}\left(1 - \binom{m + \ell}{2}\eta(H)\right)dt.
\end{equation*}
The factorials simplify because
\begin{equation*}
\binom{m + \ell}{m}\frac{H^{m + \ell}}{(m + \ell)!} = \frac{H^{m}}{m!}\frac{H^{\ell}}{\ell!}.
\end{equation*}
We now complete the finite sum inside the integral to an infinite series. With $u = H/\log t$, denote the resulting integrand by
\begin{equation*}
Q(m; u) := \frac{u^{m}}{m!}\sum_{\ell \, = \, 0}^{\infty}\frac{(-u)^{\ell}}{\ell!}\left(1 - \binom{m + \ell}{2}\eta(H)\right).
\end{equation*}
The leading term gives the series for $e^{-u}$. For the correction, expanding $(m + \ell)(m + \ell - 1)$ and shifting indices gives
\begin{equation*}
\sum_{\ell \, = \, 0}^{\infty}\binom{m + \ell}{2}\frac{(-u)^{\ell}}{\ell!} = \frac{e^{-u}}{2}\left(m(m - 1) - 2mu + u^{2}\right) = \frac{e^{-u}}{2}\left((m - u)^{2} - m\right).
\end{equation*}
Consequently,
\begin{equation*}
Q(m; u) = \frac{e^{-u}u^{m}}{m!}\left\{1 - \frac{\eta(H)}{2}\left((m - u)^{2} - m\right)\right\}.
\end{equation*}
The factor $e^{-u}u^{m}/m!$ is the probability that a Poisson random variable of mean $u$ takes the value $m$. The correction in braces comes from the second term in the singular-series approximation.

Both Bonferroni sums therefore lead to the same proposed main term, which we denote by $F(m)$:
\begin{equation*}
P(m) \approx F(m), \qquad F(m) := \frac{1}{N}\int_{2}^{N}Q\left(m; \frac{H}{\log t}\right)dt.
\end{equation*}
Finally, to obtain an expression at a single parameter, we compare this integral with $Q$ evaluated at the averaged logarithmic density:
\begin{equation*}
F(m) \approx Q(m; \mu), \qquad \mu := HJ(1) = \frac{1}{N}\int_{2}^{N}\frac{H}{\log t}dt.
\end{equation*}
This last approximation has its own error, since averaging a nonlinear function need not give its value at the averaged argument.

There are thus four contributions to keep track of: the Hardy--Littlewood residuals, the singular-series residuals, the tails introduced by completing the finite sums, and the difference between $F(m)$ and $Q(m;\mu)$. In the next section we define each residual exactly, retaining its sign. The Bonferroni sandwich proposition of Section \ref{sec:bonferroni-sandwich} then bounds $P(m) - Q(m;\mu)$ in terms of their weighted combinations, without requiring any estimates for them.

\section{Bookkeeping}
\label{sec:setup-and-bookkeeping}

For integers $H \ge 1$ and $n \ge 0$, denote the number of primes in the interval $(n, n + H]$ by
\begin{equation*}
X(n; H) := \pi(n + H) - \pi(n).
\end{equation*}
For integers $M \ge 0$, $N \ge 2$ and $m \ge 0$, define
\begin{equation*}
P(m; H, M, N) := \frac{1}{N}\left|\{M < n \le M + N,\ X(n; H) = m\}\right|.
\end{equation*}
Here $|S|$ denotes the cardinality of the finite set $S$. Thus $P(m; H, M, N)$ is the proportion of length-$H$ intervals with integer left endpoint in $(M, M + N]$ that contain exactly $m$ primes.

\subsection{Prime tuples and the Hardy--Littlewood residual}
\label{subsec:hardy-littlewood-residual}

The \emph{singular series}\footnote{Hardy and Littlewood introduced it as a series and then derived its product representation \cite[pp.~55--61]{HL1923}.} associated with a finite set of integers $\mathcal{H}$ is
\begin{equation*}
\mathfrak{S}(\mathcal{H}) := \prod_{p}\left(1 - \frac{|\mathcal{H} \bmod p|}{p}\right)\left(1 - \frac{1}{p}\right)^{-|\mathcal{H}|},
\end{equation*}
where the product is over all primes and
\begin{equation*}
\mathcal{H} \bmod p := \{h + p\mathbb{Z} : h \in \mathcal{H}\}
\end{equation*}
is the set of residue classes modulo $p$ represented in $\mathcal{H}$. If $|\mathcal{H}| = k$, then $|\mathcal{H} \bmod p| = k$ for all sufficiently large $p$, so the corresponding factor is $1 + O_{k}\left(1/p^2\right)$ and the product converges absolutely. If $\mathcal{H}$ is empty or a singleton, every factor is $1$; thus $\mathfrak{S}(\varnothing) = \mathfrak{S}(\{h\}) = 1$ for every integer $h$.

We call $\mathcal{H}$ \emph{admissible} if $|\mathcal{H} \bmod p| < p$ for every prime $p$, equivalently if $\mathfrak{S}(\mathcal{H}) > 0$, and \emph{inadmissible} otherwise. The empty set and every singleton set are admissible. If $\mathcal{H}$ is inadmissible, choose a prime $p$ for which $|\mathcal{H} \bmod p| = p$. Every translate consisting entirely of primes must contain $p$, so its translation parameter belongs to $\{p - h : h \in \mathcal{H}\}$. There are therefore at most $k$ such translates.

Write $\mathcal{P}$ for the set of primes and $\mathbf{1}_{\mathcal{P}}$ for its indicator. For a fixed finite set $\mathcal{H}$ of $k$ nonnegative integers, the prime $k$-tuples conjecture of Hardy and Littlewood asserts that
\begin{equation*}
\sum_{n \, = \, 1}^{N}\prod_{h \, \in \, \mathcal{H}}\mathbf{1}_{\mathcal{P}}(n + h) = \left(\mathfrak{S}(\mathcal{H}) + o(1)\right)\int_{2}^{N}\frac{dt}{(\log t)^{k}} \qquad (N \to \infty).
\end{equation*}
For inadmissible $\mathcal{H}$, this assertion is degenerate and holds trivially: the singular series vanishes and the sum is bounded. For admissible $\mathcal{H}$, the case $k = 0$ is immediate and the case $k = 1$ follows from the prime number theorem; the assertion remains conjectural for $k \ge 2$. Hardy and Littlewood deduced the admissible-case asymptotic from their Hypothesis~X \cite[Theorem~X~1]{HL1923}. 

Put $[H] := \{1, \ldots, H\}$. We take binomial coefficients with nonnegative integer arguments to be zero when the lower argument exceeds the upper argument. For every integer $k \ge 0$, define
\begin{equation*}
T(k; H, M, N) := \frac{1}{N}\sum_{n \, = \, 1}^{N}\binom{X(M + n; H)}{k}.
\end{equation*}
We can also express this as an average of prime-tuple counts:
\begin{equation*}
T(k; H, M, N) = \frac{1}{N}\sum_{\substack{\mathcal{H} \, \subseteq \, [H] \\ |\mathcal{H}| \, = \, k}}\sum_{n \, = \, 1}^{N}\prod_{h \, \in \, \mathcal{H}}\mathbf{1}_{\mathcal{P}}(M + n + h).
\end{equation*}
To see this, interchange the two sums on the right. For each fixed $n$, the product is $1$ precisely when every element of $M + n + \mathcal{H}$ is prime, and is $0$ otherwise. Summing over $\mathcal{H}$ therefore counts the ways to choose $k$ primes from the $X(M + n; H)$ primes in $(M + n, M + n + H]$, giving the binomial coefficient in the definition of $T$. This also holds for $k = 0$: the only choice is $\mathcal{H} = \varnothing$, and the empty product is $1$.

The corresponding singular-series sum and logarithmic integral are
\begin{equation*}
S(k; H) := \sum_{\substack{\mathcal{H} \, \subseteq \, [H] \\ |\mathcal{H}| \, = \, k}}\mathfrak{S}(\mathcal{H}), \qquad 
J(k; M, N) := \frac{1}{N}\int_{\max\{2, M\}}^{M + N}\frac{dt}{(\log t)^{k}}.
\end{equation*}
We define the signed Hardy--Littlewood residual by
\begin{equation}
\label{eq:hardy-littlewood-residual}
\Delta_{\mathrm{HL}}(k; H, M, N) := T(k; H, M, N) - J(k; M, N)S(k; H).
\end{equation}
Thus $\Delta_{\mathrm{HL}}$ is the sum of the individual prime-tuple residuals, divided by $N$:
\begin{equation*}
\Delta_{\mathrm{HL}}(k; H, M, N)
= \frac{1}{N}\sum_{\substack{\mathcal{H} \, \subseteq \, [H] \\ |\mathcal{H}| \, = \, k}}\left\{\sum_{n \, = \, 1}^{N}\prod_{h \, \in \, \mathcal{H}}\mathbf{1}_{\mathcal{P}}(M + n + h) 
- \mathfrak{S}(\mathcal{H})\int_{\max\{2, M\}}^{M + N}\frac{dt}{(\log t)^{k}}\right\}.
\end{equation*}
No estimate for this residual is assumed in this section; sufficient hypotheses for controlling its weighted sums are given in Subsection~\ref{subsec:sufficient-hardy-littlewood-hypotheses}. At $k = 0$, the definitions give
\begin{equation*}
\Delta_{\mathrm{HL}}(0; H, M, N) = \frac{\max\{2 - M, 0\}}{N}.
\end{equation*}

\subsection{The singular-series approximation}
\label{subsec:singular-series-approximation}

Let $\gamma$ be the Euler--Mascheroni constant, and put
\begin{equation*}
\eta(H) := \frac{\log H + \log(2\pi) + \gamma - 1}{H}.
\end{equation*}
Motivated by Montgomery and Soundararajan's singular-series average, after division by $k!$ to pass from ordered tuples to subsets, define the main term and the signed residual by
\begin{equation}
\label{eq:singular-series-residual}
G(k; H) := \frac{H^{k}}{k!}\left(1 - \binom{k}{2}\eta(H)\right), \qquad \Delta_{\mathrm{MS}}(k; H) := S(k; H) - G(k; H).
\end{equation}
In particular, $\Delta_{\mathrm{MS}}(0; H) = \Delta_{\mathrm{MS}}(1; H) = 0$. Gallagher's earlier estimate \cite[equation~(3)]{GAL1976}, translated from ordered tuples to our subset normalization, gives, for every fixed integer $k \ge 1$,
\begin{equation*}
S(k; H) \sim \frac{H^{k}}{k!} \qquad (H \to \infty).
\end{equation*}
His proof starts from the Euler product for the singular series, expands it into an absolutely convergent series, and averages the resulting congruence conditions by a lattice-point count using the Chinese remainder theorem; see in particular \cite[equations~(8)--(10)]{GAL1976}. Montgomery and Soundararajan's estimate \cite[equation~(17)]{MS2004} sharpens this first-order average by identifying the secondary term in \eqref{eq:singular-series-residual}. In our normalization, for every fixed integer $k \ge 2$ and every $\epsilon > 0$, it gives
\begin{equation*}
\Delta_{\mathrm{MS}}(k; H) \ll_{k,\epsilon} H^{k - 3/2 + \epsilon}.
\end{equation*}
Where Gallagher's proof is essentially an Euler-product and lattice-point averaging argument, Montgomery and Soundararajan pass to centered singular-series sums and relate them to moments of reduced residues. The case $k = 2$ uses Goldston's sharper pair average, which supplies the constant term in $\eta(H)$ \cite[equations~(47)--(48)]{MS2004}.

We state this estimate to explain the choice of main terms; the argument below will require control of weighted sums in which $k$ is allowed to grow with $H$. The definition of $\Delta_{\mathrm{MS}}$ applies to every integer $k \ge 0$, including $k > H$, when $T(k; H, M, N) = S(k; H) = 0$. Proposition~\ref{prop:inclusion-exclusion-sandwich} uses only the definition of this residual. The estimates needed for its weighted sums will be proved unconditionally in Section~\ref{sec:singular-series-finite-sieve}.

\subsection{The corrected Poisson expression and averaging}
\label{subsec:corrected-poisson-and-averaging}

For $u \ge 0$, put
\begin{equation}
\label{eq:corrected-poisson-expression}
Q(m; u, H) := \frac{e^{-u}u^{m}}{m!}\left\{1 - \frac{\eta(H)}{2}\left((m - u)^{2} - m\right)\right\},
\end{equation}
where the value at $u = 0$ is defined by continuity. Thus $Q(0; 0, H) = 1$ and $Q(m; 0, H) = 0$ for $m \ge 1$. The first factor is the probability that a Poisson random variable of mean $u$ takes the value $m$. As we shall see in Subsection~\ref{subsec:truncation-residuals}, the leading term in the singular-series approximation produces this Poisson factor under inclusion--exclusion, while the term involving $\eta(H)$ produces the correction in braces. We average the whole expression over the varying logarithmic density:
\begin{equation}
\label{eq:integrated-main-term}
F(m; H, M, N) := \frac{1}{N}\int_{\max\{2, M\}}^{M + N}Q\left(m; \frac{H}{\log t}, H\right)dt.
\end{equation}

To compare this average with $Q$ evaluated at a single parameter, put
\begin{equation*}
\mu = \mu(H, M, N) := H J(1; M, N) = \frac{1}{N}\int_{\max\{2, M\}}^{M + N}\frac{H}{\log t}dt
\end{equation*}
and define the signed averaging residual by
\begin{equation}
\label{eq:density-averaging-residual}
\Delta_{\mathrm{av}}(m; H, M, N) := F(m; H, M, N) - Q(m; \mu, H).
\end{equation}
The parameter $\mu$ retains the logarithmic integral exactly; no approximation involving $\log N$ has been made. The exact mean count is $T(1; H, M, N)$. Interchanging the sums and shifting the endpoints gives
\begin{equation*}
T(1; H, M, N) = \frac{1}{N}\sum_{h \, = \, 1}^{H}\left(\pi(M + N + h) - \pi(M + h)\right) = \frac{H}{N}\left(\pi(M + N) - \pi(M)\right) + O\left(\frac{H^{2}}{N}\right).
\end{equation*}
Thus estimates for the prime count in $(M, M + N]$ give corresponding estimates for the difference between the mean count and $\mu$. In particular, when $0 \le M \le C N$ for a fixed constant $C > 0$ and $H \asymp \log N$, the quantitative prime number theorem gives $T(1; H, M, N) - \mu = o(1/H)$ uniformly in $M$. This provides a natural range for the later asymptotic results, although the definitions and the Bonferroni sandwich require no such restriction.

\subsection{Truncation residuals}
\label{subsec:truncation-residuals}

In the proof of the Bonferroni sandwich proposition below, we shall encounter finite alternating sums. The following series representation of $Q$ will allow us to identify their main terms:
\begin{equation*}
Q(m; u, H) = \frac{u^{m}}{m!}\sum_{\ell \, = \, 0}^{\infty}\frac{(-u)^{\ell}}{\ell!}\left(1 - \binom{m + \ell}{2}\eta(H)\right).
\end{equation*}
Indeed, expanding $(m + \ell)(m + \ell - 1)$ and shifting indices in the exponential series gives
\begin{equation*}
\sum_{\ell \, = \, 0}^{\infty}\binom{m + \ell}{2}\frac{(-u)^{\ell}}{\ell!} = \frac{e^{-u}}{2}\left(m(m - 1) - 2mu + u^{2}\right) = \frac{e^{-u}}{2}\left((m - u)^{2} - m\right).
\end{equation*}

As we shall see, substituting the expression subtracted in the definition of $\Delta_{\mathrm{MS}}$ into those alternating sums produces truncations of the series representing $Q$. To record the difference between the sum through $\ell = j$ and $Q$, define the signed truncation residual
\begin{equation}
\label{eq:exponential-tail-residual}
\Delta_{\mathrm{tr}}(m, j; u, H) := -\frac{u^{m}}{m!}\sum_{\ell \, > \, j}\frac{(-u)^{\ell}}{\ell!}\left(1 - \binom{m + \ell}{2}\eta(H)\right)
\end{equation}
for integers $m, j \ge 0$. Thus
\begin{equation*}
\frac{u^{m}}{m!}\sum_{\ell \, = \, 0}^{j}\frac{(-u)^{\ell}}{\ell!}\left(1 - \binom{m + \ell}{2}\eta(H)\right) = Q(m; u, H) + \Delta_{\mathrm{tr}}(m, j; u, H).
\end{equation*}
The series converges absolutely and locally uniformly for $u \ge 0$, and $\Delta_{\mathrm{tr}}(m, j; 0, H) = 0$, with values at $u = 0$ understood by continuity.

\subsection{Combining the residuals}
\label{subsec:combining-the-residuals}

The Bonferroni sandwich will involve the following weighted sums of the Hardy--Littlewood and Montgomery--Soundararajan residuals, together with the averaged truncation residual:

Put
\begin{align*}
R_{\mathrm{HL}}(m, j; H, M, N) & := \sum_{\ell \, = \, 0}^{j}(-1)^{\ell}\binom{m + \ell}{m}\Delta_{\mathrm{HL}}(m + \ell; H, M, N), \\[1ex]
R_{\mathrm{MS}}(m, j; H, M, N) & := \sum_{\ell \, = \, 0}^{j}(-1)^{\ell}\binom{m + \ell}{m}J(m + \ell; M, N)\Delta_{\mathrm{MS}}(m + \ell; H), \\
\intertext{and}
R_{\mathrm{tr}}(m, j; H, M, N) & := \frac{1}{N}\int_{\max\{2, M\}}^{M + N}\Delta_{\mathrm{tr}}\left(m, j; \frac{H}{\log t}, H\right)dt.
\end{align*}
Their sum is denoted by
\begin{equation*}
\mathcal{R}(m, j; H, M, N) := R_{\mathrm{HL}}(m, j; H, M, N) + R_{\mathrm{MS}}(m, j; H, M, N) + R_{\mathrm{tr}}(m, j; H, M, N).
\end{equation*}

\section{A Bonferroni sandwich}
\label{sec:bonferroni-sandwich}

We now combine the identities of Section~\ref{sec:setup-and-bookkeeping} with the Bonferroni inequalities. Recall that $P(m; H, M, N)$ is the proportion of intervals containing exactly $m$ primes, and $F(m; H, M, N)$ is the integrated corrected Poisson expression. The combined residual $\mathcal{R}$ is the sum of the weighted Hardy--Littlewood, singular-series and truncation residuals defined in Subsection~\ref{subsec:combining-the-residuals}. The bounds below retain their signs and require no estimates for them.

\begin{proposition}[Bonferroni sandwich]
\label{prop:inclusion-exclusion-sandwich}
Let $H \ge 1$, $M \ge 0$, $N \ge 2$ and $m \ge 0$ be integers, and let $L \ge 0$ be an even integer. Throughout the statement and proof, we suppress the dependence on $H,M,N$ in the notation and write $\eta = \eta(H)$. Then
\begin{equation}
\label{eq:integrated-sandwich}
\mathcal{R}(m, L + 1) \le P(m) - F(m) \le \mathcal{R}(m, L).
\end{equation}
Equivalently,
\begin{equation}
\label{eq:mean-adjusted-sandwich}
\mathcal{R}(m, L + 1) \le P(m) - Q(m; \mu) - \Delta_{\mathrm{av}}(m) \le \mathcal{R}(m, L).
\end{equation}
In particular, if
\begin{equation*}
E(m, L) := \max_{j \, \in \, \{L, L + 1\}}\left|\Delta_{\mathrm{av}}(m) + \mathcal{R}(m, j)\right|,
\end{equation*}
then
\begin{equation}
\label{eq:sandwich-absolute-bound}
\left|P(m) - Q(m; \mu)\right| \le E(m, L).
\end{equation}
The separate residuals give the further bound
\begin{equation*}
E(m, L) \le |\Delta_{\mathrm{av}}(m)| + \max_{j \, \in \, \{L, L + 1\}}\left\{|R_{\mathrm{HL}}(m, j)| + |R_{\mathrm{MS}}(m, j)| + |R_{\mathrm{tr}}(m, j)|\right\}.
\end{equation*}
All these statements are unconditional.
\end{proposition}

\begin{proof}
For each integer $j \ge 0$, put
\begin{equation*}
B(m, j) := \sum_{\ell \, = \, 0}^{j}(-1)^{\ell}\binom{m + \ell}{m}T(m + \ell).
\end{equation*}
We first compare this partial sum with $P(m)$. For integers $x > m$,
\begin{equation*}
\sum_{\ell \, = \, 0}^{j}(-1)^{\ell}\binom{m + \ell}{m}\binom{x}{m + \ell} = \binom{x}{m}\sum_{\ell \, = \, 0}^{j}(-1)^{\ell}\binom{x - m}{\ell} = (-1)^{j}\binom{x}{m}\binom{x - m - 1}{j}.
\end{equation*}
The last equality follows by telescoping Pascal's identity. For $x = m$ the sum is $1$, and for $0 \le x < m$ it is $0$. Substituting $x = X(M + n)$ and averaging over $1 \le n \le N$ gives
\begin{equation}
\label{eq:bonferroni-exact-remainder}
B(m, j) - P(m) = \frac{(-1)^{j}}{N}\sum_{\substack{n \, = \, 1 \\ X(M + n) \, \ge \, m + j + 1}}^{N}\binom{X(M + n)}{m}\binom{X(M + n) - m - 1}{j}.
\end{equation}
Every summand is nonnegative. Since $L$ is even, \eqref{eq:bonferroni-exact-remainder} implies
\begin{equation}
\label{eq:bonferroni-sandwich}
B(m, L + 1) \le P(m) \le B(m, L).
\end{equation}

We now separate the proposed main terms from the residuals. By \eqref{eq:hardy-littlewood-residual} and \eqref{eq:singular-series-residual},
\begin{equation*}
T(k) = J(k)G(k) + J(k)\Delta_{\mathrm{MS}}(k) + \Delta_{\mathrm{HL}}(k).
\end{equation*}
Consequently,
\begin{equation}
\label{eq:partial-sum-decomposition}
B(m, j) - R_{\mathrm{HL}}(m, j) - R_{\mathrm{MS}}(m, j) = \sum_{\ell \, = \, 0}^{j}(-1)^{\ell}\binom{m + \ell}{m}J(m + \ell)G(m + \ell).
\end{equation}
The series representation of $Q$ in Section~\ref{sec:setup-and-bookkeeping} and the definition \eqref{eq:exponential-tail-residual} give
\begin{equation}
\label{eq:truncated-main-term-identity}
\frac{u^{m}}{m!}\sum_{\ell \, = \, 0}^{j}\frac{(-u)^{\ell}}{\ell!}\left(1 - \binom{m + \ell}{2}\eta\right) = Q(m; u) + \Delta_{\mathrm{tr}}(m, j; u).
\end{equation}
By the definition of $G(k)$ in \eqref{eq:singular-series-residual} and the definition of $J(k)$, the right side of \eqref{eq:partial-sum-decomposition} is the integral of the left side of \eqref{eq:truncated-main-term-identity}, with $u = H/\log t$, over $\max\{2, M\} \le t \le M + N$, divided by $N$. The sum is finite, so interchanging it with the integral requires no limiting argument. Using \eqref{eq:integrated-main-term} and the definition of $R_{\mathrm{tr}}$, we obtain
\begin{equation}
\label{eq:exact-partial-sum-expansion}
B(m, j) = F(m) + \mathcal{R}(m, j).
\end{equation}
Combining \eqref{eq:exact-partial-sum-expansion} with \eqref{eq:bonferroni-sandwich} proves \eqref{eq:integrated-sandwich}. Substitution of \eqref{eq:density-averaging-residual} gives \eqref{eq:mean-adjusted-sandwich}. Thus $P(m) - Q(m; \mu)$ lies between $\Delta_{\mathrm{av}}(m) + \mathcal{R}(m, L + 1)$ and $\Delta_{\mathrm{av}}(m) + \mathcal{R}(m, L)$, proving \eqref{eq:sandwich-absolute-bound}. The final bound follows by the triangle inequality.
\end{proof}

\section{The singular series through a finite sieve}
\label{sec:singular-series-finite-sieve}

We estimate the signed sums of singular series that occur in the Bonferroni sandwich (Proposition \ref{prop:inclusion-exclusion-sandwich}). The argument uses a pair average of Goldston \cite{GOL1990} and a fourth-moment estimate of Montgomery and Soundararajan \cite{MS2004}. An elementary upper bound for the number of survivors of a finite sieve supplies the uniformity needed for the remaining terms. In particular, we never apply a fixed-$k$ asymptotic with $k$ tending to infinity.

Recall from Section~\ref{sec:setup-and-bookkeeping} that
\begin{equation*}
S(k; H) = \sum_{\substack{\mathcal{H} \, \subseteq \, [H] \\ |\mathcal{H}| \, = \, k}}\mathfrak{S}(\mathcal{H}).
\end{equation*}
The sum is over subsets, so each set of shifts is counted once, and $S(k; H) = 0$ for $k > H$. The corrected Poisson expression is
\begin{equation*}
Q(m; u, H) = \frac{e^{-u}u^{m}}{m!}\left\{1 - \frac{\eta(H)}{2}\left((m - u)^{2} - m\right)\right\}, \qquad \eta(H) = \frac{\log H + \log(2\pi) + \gamma - 1}{H}.
\end{equation*}
We shall compare $Q$ with a finite alternating sum of the singular-series averages. For integers $H \ge 2$ and $m, j \ge 0$, and a real number $u \ge 0$, write
\begin{equation*}
\mathcal{A}(m, j; u, H) := \sum_{\ell \, = \, 0}^{j}(-1)^{\ell}\binom{m + \ell}{m}\left(\frac{u}{H}\right)^{m + \ell}S(m + \ell; H).
\end{equation*}
Values at $u = 0$ are understood by continuity. Substituting $u = H/\log t$ and averaging over the starting points will connect this estimate to the Bonferroni sandwich in Subsection~\ref{subsec:integrated-singular-series-residuals}.

\begin{theorem}
\label{thm:weighted-singular-series-asymptotic}
There is an absolute constant $C \ge 1$ such that, for every fixed $U > 0$ and $\epsilon > 0$, uniformly for $H \ge 2$, $0 \le u \le U$ and integers $m, j \ge 0$,
\begin{equation}
\label{eq:weighted-singular-series-asymptotic}
\mathcal{A}(m, j; u, H) = Q(m; u, H) + O_{U, \epsilon}\left(2^{-m}H^{-3/2 + \epsilon} + \frac{(Cu)^{m + j + 1}}{m!(j + 1)!}\right).
\end{equation}
The implied constant is independent of $m$ and $j$.
\end{theorem}

\begin{corollary}
\label{cor:singular-series-growing-truncation}
Fix an integer $m \ge 0$ and real numbers $U, \delta > 0$. Let $L$ be the least even integer greater than or equal to $(1 + \delta)\log H/\log\log H$. As $H \to \infty$, uniformly for $0 \le u \le U$ and $j \in \{L, L + 1\}$,
\begin{equation*}
\mathcal{A}(m, j; u, H) = Q(m; u, H) + O_{m, U, \delta, \epsilon}\left(H^{-3/2 + \epsilon} + H^{-1 - \delta + o(1)}\right).
\end{equation*}
In particular, the error is $o(1/H)$ if $0 < \epsilon < 1/2$.
\end{corollary}

\begin{proof}[Proof of Corollary~\ref{cor:singular-series-growing-truncation}]
Apply Theorem~\ref{thm:weighted-singular-series-asymptotic}, noting that by Stirling's formula,
\begin{equation*}
\log\left(\frac{(CU)^{m + j + 1}}{m!(j + 1)!}\right) = -(j + 1)\log(j + 1) + O_{m, U}(j + 1) = -(1 + \delta)\log H + o(\log H).
\end{equation*}
\end{proof}

\subsection{A finite probability space}
\label{subsec:finite-sieve-probability-space}

Throughout the proof put
\begin{equation*}
y := H^{3}, \qquad q := \prod_{p \, \le \, y}p, \qquad \rho := \frac{\phi(q)}{q} = \prod_{p \, \le \, y}\left(1 - \frac{1}{p}\right),
\end{equation*}
where $\phi$ is Euler's totient function. Mertens' product estimate gives $\rho \asymp 1/\log H$, with absolute implied constants. Choose $a$ uniformly from the residue classes modulo $q$, and let
\begin{equation*}
Y(a) := \left|\{h \in [H] : \gcd(a + h, q) = 1\}\right|.
\end{equation*}
We write $\mathbb{P}$ and $\mathbb{E}$ for probability and expectation on this finite space. Equivalently, for each prime $p \le y$, independently choose one residue class to exclude from $[H]$. The Chinese remainder theorem identifies these choices with $-a \bmod q$. These probabilities express finite averages of explicitly defined sets; they make no assumption about the distribution of primes.

For a subset $\mathcal{H}$ of $[H]$, define the finite singular series and its sum by
\begin{equation*}
\mathfrak{S}_{y}(\mathcal{H}) := \prod_{p \, \le \, y}\left(1 - \frac{|\mathcal{H} \bmod p|}{p}\right)\left(1 - \frac{1}{p}\right)^{-|\mathcal{H}|}, \qquad S_{y}(k; H) := \sum_{\substack{\mathcal{H} \, \subseteq \, [H] \\ |\mathcal{H}| \, = \, k}}\mathfrak{S}_{y}(\mathcal{H}).
\end{equation*}
Both empty-set values are $1$, and $S_{y}(k; H) = 0$ for $k > H$.

\begin{lemma}
\label{lem:finite-sieve-factorial-moments}
For every integer $k \ge 0$,
\begin{equation}
\label{eq:finite-sieve-factorial-moments}
\mathbb{E}\binom{Y}{k} = \rho^{k}S_{y}(k; H).
\end{equation}
In particular, $\mathbb{E}Y = H\rho$.
\end{lemma}

\begin{proof}
For a fixed $k$-element set $\mathcal{H}$, the probability that all its elements survive is
\begin{equation*}
\mathbb{P}\left(\gcd(a + h, q) = 1\text{ for every }h \in \mathcal{H}\right) = \prod_{p \, \le \, y}\left(1 - \frac{|\mathcal{H} \bmod p|}{p}\right) = \rho^{k}\mathfrak{S}_{y}(\mathcal{H}).
\end{equation*}
Here and below, ordinary probabilities are averages over the $q$ possible values of $a$. The binomial coefficient counts the $k$-element sets of survivors, so summing these probabilities proves \eqref{eq:finite-sieve-factorial-moments}.
\end{proof}

\begin{lemma}
\label{lem:finite-euler-product-comparison}
Uniformly for $0 \le k \le H$,
\begin{equation}
\label{eq:finite-euler-product-comparison}
0 \le S_{y}(k; H) - S(k; H) \ll \frac{k^{2}}{y}S_{y}(k; H).
\end{equation}
The implied constant is absolute, and the difference is zero for $k = 0, 1$.
\end{lemma}

\begin{proof}
If $p > y > H$, all elements of a $k$-element subset of $[H]$ represent distinct residue classes modulo $p$. Consequently,
\begin{equation*}
\mathfrak{S}(\mathcal{H}) = \mathfrak{S}_{y}(\mathcal{H})\,t(k; y), \qquad t(k; y) := \prod_{p \, > \, y}\left(1 - \frac{k}{p}\right)\left(1 - \frac{1}{p}\right)^{-k}.
\end{equation*}
For $k = 0, 1$ the last product is $1$. For $2 \le k \le H$, expand the logarithms to obtain
\begin{equation*}
\log t(k; y) = -\sum_{p \, > \, y}\sum_{r \, = \, 2}^{\infty}\frac{k^{r} - k}{r p^{r}}, \qquad 0 \le -\log t(k; y) \ll k^{2}\sum_{n \, > \, y}\frac{1}{n^{2}} \ll \frac{k^{2}}{y}.
\end{equation*}
The estimates are uniform because $k/p \le H/y \le 1/4$. It follows that 
\begin{equation*}
0 \le 1 - t(k; y) \le -\log t(k; y) \ll \frac{k^{2}}{y}.
\end{equation*}
Multiply by the nonnegative finite singular series and sum. The same argument includes inadmissible sets, for which both singular series vanish.
\end{proof}

\subsection{An upper bound uniform in the number of shifts}
\label{subsec:uniform-sieve-upper-bound}

We need only an upper bound for the high factorial moments. The following elementary form of the Selberg upper-bound sieve is sufficient; see \cite[Section~3.2]{MV2006}. We include the proof to make the uniformity in the translate explicit and to record its consequence for all factorial moments.

\begin{lemma}
\label{lem:uniform-survivor-bound}
There is an absolute constant $C_{0} \ge 1$ such that
\begin{equation}
\label{eq:uniform-survivor-bound}
Y(a) \le C_{0}H\rho
\end{equation}
for every residue class $a \bmod q$ and every integer $H \ge 2$. Consequently, for every integer $k \ge 0$,
\begin{equation}
\label{eq:uniform-singular-series-upper-bound}
0 \le S(k; H) \le S_{y}(k; H) \le \frac{(C_{0}H)^{k}}{k!}.
\end{equation}
\end{lemma}

\begin{proof}
It suffices first to take $H$ large. Put $z = H^{1/8}$, and let $\mathcal{D}$ be the set of squarefree positive integers less than or equal to $z$. Every member of $\mathcal{D}$ divides $q$. Write $\omega(d)$ for the number of distinct prime divisors of $d$, and put\footnote{The coefficients $\lambda_{d}$ are sieve weights used only in this proof; they are unrelated to the Poisson parameter $\lambda$ used elsewhere.}
\begin{equation*}
\mathcal{T} := \sum_{r \, \in \, \mathcal{D}}\frac{1}{\phi(r)}, \qquad b_{r} := \frac{(-1)^{\omega(r)}}{\phi(r)\mathcal{T}}, \qquad \lambda_{d} := d\sum_{\substack{r \, \in \, \mathcal{D} \\ d \, \mid \, r}}(-1)^{\omega(r/d)}b_{r} \quad (d \in \mathcal{D}).
\end{equation*}
Then $\lambda_{1} = 1$ and $|\lambda_{d}| \le d$. Inversion over the squarefree divisors gives
\begin{equation*}
\sum_{\substack{d \, \in \, \mathcal{D} \\ r \, \mid \, d}}\frac{\lambda_{d}}{d} = b_{r} \qquad (r \in \mathcal{D}).
\end{equation*}
Indeed, after substituting the definition of $\lambda_{d}$, the coefficient of $b_{s}$ is the sum of $(-1)^{\omega(s/d)}$ over $r \mid d \mid s$, which is $1$ if $s = r$ and $0$ otherwise. Using $\gcd(d, e) = \sum_{r \, \mid \, d,\ r \, \mid \, e}\phi(r)$, we therefore have
\begin{equation*}
\sum_{d, e \, \in \, \mathcal{D}}\frac{\lambda_{d}\lambda_{e}}{\operatorname{lcm}(d, e)} = \sum_{r \, \in \, \mathcal{D}}\phi(r)\left(\sum_{\substack{d \, \in \, \mathcal{D} \\ r \, \mid \, d}}\frac{\lambda_{d}}{d}\right)^{2} = \frac{1}{\mathcal{T}}.
\end{equation*}
If $h$ survives, the sum of $\lambda_{d}$ over $d \in \mathcal{D}$ dividing $a + h$ is $\lambda_{1} = 1$. Its square is nonnegative for all other $h$. Counting multiples of $\operatorname{lcm}(d, e)$ in an interval of $H$ consecutive integers now gives
\begin{equation*}
Y(a) \le \sum_{h \, = \, 1}^{H}\left(\sum_{\substack{d \, \in \, \mathcal{D} \\ d \, \mid \, a + h}}\lambda_{d}\right)^{2} = \frac{H}{\mathcal{T}} + O\left(\left(\sum_{d \, \in \, \mathcal{D}}|\lambda_{d}|\right)^{2}\right) = \frac{H}{\mathcal{T}} + O(z^{4}).
\end{equation*}
All constants here are independent of $a$.

For completeness, $\mathcal{T} \gg \log z$. The number of nonsquarefree integers less than or equal to $x$ is at most $x\sum_{r \, = \, 2}^{\infty}r^{-2}$, and this sum is less than $1$. Thus the number of squarefree integers less than or equal to $x$ is bounded below by a positive constant times $x$, apart from an absolute additive constant. Partial summation gives $\sum_{r \, \in \, \mathcal{D}}1/r \gg \log z$, and $1/\phi(r) \ge 1/r$. Since $z^{4} = H^{1/2}$, the preceding estimate proves $Y(a) \ll H/\log H \ll H\rho$. Increasing an absolute constant covers the remaining values of $H$.

Finally, \eqref{eq:finite-sieve-factorial-moments} and \eqref{eq:uniform-survivor-bound} give
\begin{equation*}
\rho^{k}S_{y}(k; H) = \mathbb{E}\binom{Y}{k} \le \frac{\mathbb{E}Y^{k}}{k!} \le \frac{(C_{0}H\rho)^{k}}{k!}.
\end{equation*}
Together with Lemma~\ref{lem:finite-euler-product-comparison}, this proves \eqref{eq:uniform-singular-series-upper-bound}; for $k > H$ both sums are zero.
\end{proof}

\subsection{The second and fourth moments}
\label{subsec:sieve-second-fourth-moments}

Let $b := H\rho$ be the mean of $Y$, and write 
\begin{equation*}
v := \mathbb{E}(Y - b)^{2}
\end{equation*}
for its variance. The two moment estimates below use Goldston's pair average \cite[equation~(33)]{GOL1990}, also recorded in \cite[equations~(47)--(48)]{MS2004}, and the fourth-moment case of Montgomery and Soundararajan's Theorem~1 \cite{MS2004}. All moment orders used in these inputs are fixed.

\begin{lemma}
\label{lem:sieve-low-moments}
For every $\epsilon > 0$,
\begin{equation}
\label{eq:sieve-variance-correction}
v - b = -\rho^{2}H^{2}\eta(H) + O_{\epsilon}\left(\rho^{2}H^{1/2 + \epsilon}\right).
\end{equation}
Moreover, with absolute implied constants,
\begin{equation}
\label{eq:sieve-moment-bounds}
v \ll H\rho, \qquad \mathbb{E}(Y - b)^{4} \ll H^{2}\rho^{2}, \qquad \mathbb{E}|Y - b|^{3} \ll (H\rho)^{3/2}.
\end{equation}
\end{lemma}

\begin{proof}
Goldston's pair average \cite[equation~(33)]{GOL1990} gives
\begin{equation*}
2S(2; H) = 2\sum_{d \, = \, 1}^{H}(H - d)\mathfrak{S}(\{0,d\}) = H^{2}\left(1 - \eta(H)\right) + O_{\epsilon}(H^{1/2 + \epsilon}).
\end{equation*}
The first equality follows by grouping pairs of shifts according to their difference; the same form of the estimate appears in \cite[equation~(48)]{MS2004}. By \eqref{eq:finite-euler-product-comparison} and \eqref{eq:uniform-singular-series-upper-bound}, replacing $S(2; H)$ by $S_{y}(2; H)$ introduces an error $O(H^{2}/y) = O(H^{-1})$. Since $Y^{2} = Y(Y - 1) + Y$, the identity \eqref{eq:finite-sieve-factorial-moments} gives
\begin{equation*}
v - b = \mathbb{E}Y(Y - 1) - b^{2} = \rho^{2}\left(2S_{y}(2; H) - H^{2}\right).
\end{equation*}
This proves \eqref{eq:sieve-variance-correction}. Taking, for example, $\epsilon = 1/4$ there, and recalling $\rho \asymp 1/\log H$, also proves $v \ll H\rho$.

In the notation of \cite[equations~(10)--(11)]{MS2004}, the normalized centered moments of $Y$ are $\mathbb{E}(Y - b)^{r} = \rho^{r}V_{r}(q; H)$. In particular, $V_{2}(q; H) = v/\rho^{2}$. Their Theorem~1, used only at $r = 4$, reads
\begin{equation*}
V_{4}(q; H) = 3V_{2}(q; H)^{2} + O\left(H^{2 - 1/28}\rho^{-18}\right).
\end{equation*}
Multiplying by $\rho^{4}$ gives
\begin{equation*}
\mathbb{E}(Y - b)^{4} = 3v^{2} + O\left(H^{2 - 1/28}\rho^{-14}\right) \ll H^{2}\rho^{2}.
\end{equation*}
Indeed, the ratio of the error term to $H^{2}\rho^{2}$ is $O(H^{-1/28}(\log H)^{16})$, which tends to zero. This proves the bound for all sufficiently large $H$; increasing the absolute implied constant covers the remaining integers $H \ge 2$. Finally, H\"older's inequality gives $\mathbb{E}|Y - b|^{3} \le (\mathbb{E}(Y - b)^{4})^{3/4}$. This proves \eqref{eq:sieve-moment-bounds}.
\end{proof}

\subsection{Thinning the survivors}
\label{subsec:thinning-sieve-survivors}

For fixed $U > 0$ and sufficiently large $H$, uniformly for $0 \le u \le U$, the number
\begin{equation*}
\theta := \frac{u}{H\rho} = \frac{u}{b}
\end{equation*}
lies in $[0, 1]$. Independently retain each of the $Y$ survivors with probability $\theta$, and call the number retained $Z$. Conditional on $Y$, the variable $Z$ has the binomial distribution with parameters $Y$ and $\theta$. In particular, $\mathbb{E}Z = u$ and, for every complex number $z$,
\begin{equation}
\label{eq:thinned-sieve-generating-function}
\mathbb{E}z^{Z} = \mathbb{E}\left(1 + \theta(z - 1)\right)^{Y} = \sum_{k \, = \, 0}^{H}S_{y}(k; H)\left(\frac{u}{H}\right)^{k}(z - 1)^{k}.
\end{equation}
The second equality follows by the binomial theorem and \eqref{eq:finite-sieve-factorial-moments}. This identity converts an entire sum over tuple sizes into one function of $Y$.

\begin{proposition}
\label{prop:singular-series-generating-function}
Define
\begin{equation*}
\mathcal{G}(z; u, H) := \sum_{k \, = \, 0}^{H}S(k; H)\left(\frac{u}{H}\right)^{k}(z - 1)^{k}.
\end{equation*}
For every fixed $U, R, \epsilon > 0$, uniformly for $0 \le u \le U$ and complex $z$ with $|z| \le R$,
\begin{equation}
\label{eq:singular-series-generating-function-expansion}
\mathcal{G}(z; u, H) = e^{u(z - 1)}\left(1 - \frac{\eta(H)u^{2}}{2}(z - 1)^{2}\right) + O_{U, R, \epsilon}(H^{-3/2 + \epsilon}).
\end{equation}
For sufficiently large $H$, depending on $U$, the random variable $Z$ is defined for every $0 \le u \le U$, and the same expansion holds with $\mathcal{G}(z; u, H)$ replaced by $\mathbb{E}z^{Z}$.
\end{proposition}

\begin{proof}
Put $w = z - 1$, $D = Y - b$, and $s = \log(1 + \theta w)$, using the power series for the logarithm near $1$. For $H$ sufficiently large in terms of $U$ and $R$, we have $|\theta w| \le 1/2$. Taylor expansion gives
\begin{equation*}
s = \theta w - \frac{\theta^{2}w^{2}}{2} + O_{R}(\theta^{3}), \qquad bs = uw - \frac{b\theta^{2}w^{2}}{2} + O_{U, R}(b^{-2}), \qquad s^{2} = \theta^{2}w^{2} + O_{R}(\theta^{3}).
\end{equation*}
Here $b\theta = u$ and $b\theta^{3} \ll_{U} b^{-2}$. The upper bound \eqref{eq:uniform-survivor-bound} implies $|D| \le (C_{0} + 1)b$, and hence $|sD| \ll_{U, R} 1$. The remainder in the expansion of $e^{sD}$ is therefore bounded by a constant times $|sD|^{3}$, uniformly for every outcome. Since $\mathbb{E}D = 0$, Lemma~\ref{lem:sieve-low-moments} gives
\begin{equation*}
\mathbb{E}e^{sD} = 1 + \frac{s^{2}v}{2} + O_{U, R}\left(|s|^{3}\mathbb{E}|D|^{3}\right) = 1 + \frac{\theta^{2}w^{2}v}{2} + O_{U, R}(b^{-3/2}).
\end{equation*}
The replacement of $s^{2}$ by $\theta^{2}w^{2}$ costs $O_{U, R}(\theta^{3}v) = O_{U, R}(b^{-2})$, since $v \ll b$. Similarly,
\begin{equation*}
e^{bs} = e^{uw}\left(1 - \frac{b\theta^{2}w^{2}}{2} + O_{U, R}(b^{-2})\right).
\end{equation*}
Multiplying these two expansions in $\mathbb{E}z^{Z} = e^{bs}\mathbb{E}e^{sD}$ yields
\begin{equation}
\label{eq:thinning-variance-minus-mean}
\mathbb{E}z^{Z} = e^{uw}\left(1 + \frac{\theta^{2}(v - b)}{2}w^{2}\right) + O_{U, R}(b^{-3/2}).
\end{equation}
The cross term has size $O_{U, R}(b^{-2})$. This is the cancellation on which the argument depends: the second term from the binomial thinning subtracts the mean $b$ from the variance $v$.

Now \eqref{eq:sieve-variance-correction} gives
\begin{equation*}
\theta^{2}(v - b) = -u^{2}\eta(H) + O_{U, \epsilon}(H^{-3/2 + \epsilon}).
\end{equation*}
Also, $b^{-3/2} \ll (\log H/H)^{3/2} \ll_{\epsilon} H^{-3/2 + \epsilon}$. Substitution in \eqref{eq:thinning-variance-minus-mean} proves the assertion for $\mathbb{E}z^{Z}$.

It remains to restore the omitted Euler factors. By \eqref{eq:finite-euler-product-comparison}, \eqref{eq:uniform-singular-series-upper-bound} and \eqref{eq:thinned-sieve-generating-function},
\begin{equation*}
\left|\mathcal{G}(z; u, H) - \mathbb{E}z^{Z}\right| \ll \frac{1}{y}\sum_{k \, = \, 0}^{H}\frac{k^{2}(C_{0}u|z - 1|)^{k}}{k!} \ll_{U, R} \frac{1}{y} = H^{-3}.
\end{equation*}
This is smaller than the asserted error and proves \eqref{eq:singular-series-generating-function-expansion}.
\end{proof}

\begin{proof}[Proof of Theorem~\ref{thm:weighted-singular-series-asymptotic}]
Apply Proposition~\ref{prop:singular-series-generating-function} on the circle $|z| = 2$. Cauchy's coefficient formula bounds the coefficient of $z^{m}$ in the error by $O_{U, \epsilon}(2^{-m}H^{-3/2 + \epsilon})$, uniformly for every $m \ge 0$. The coefficient of $z^{m}$ in $\mathcal{G}$ is
\begin{equation*}
\sum_{\ell \, = \, 0}^{\infty}(-1)^{\ell}\binom{m + \ell}{m}\left(\frac{u}{H}\right)^{m + \ell}S(m + \ell; H),
\end{equation*}
a finite sum because $S(k; H) = 0$ for $k > H$. The coefficient in the main expression in \eqref{eq:singular-series-generating-function-expansion} is $Q(m; u, H)$. To see this without exceptions at $m = 0, 1$, observe that
\begin{equation*}
\sum_{m \, = \, 0}^{\infty}Q(m; u, H)z^{m} = e^{u(z - 1)} - \frac{\eta(H)u^{2}}{2}\frac{\partial^{2}}{\partial u^{2}}e^{u(z - 1)} = e^{u(z - 1)}\left(1 - \frac{\eta(H)u^{2}}{2}(z - 1)^{2}\right).
\end{equation*}
For $u > 0$, differentiating $e^{-u}u^{m}/m!$ twice gives the defining factor $((m - u)^{2} - m)/u^{2}$; continuity handles $u = 0$.

Finally, \eqref{eq:uniform-singular-series-upper-bound} bounds the part with $\ell > j$ in absolute value by
\begin{equation*}
\frac{(C_{0}u)^{m}}{m!}\sum_{\ell \, > \, j}\frac{(C_{0}u)^{\ell}}{\ell!} \le e^{C_{0}u}\frac{(C_{0}u)^{m + j + 1}}{m!(j + 1)!}.
\end{equation*}
Here we used $(j + 1 + r)! \ge (j + 1)!r!$ for $r \ge 0$. Subtracting this tail proves \eqref{eq:weighted-singular-series-asymptotic}, with any absolute $C \ge C_{0}$. The proof applies for sufficiently large $H$ in terms of $U$; increasing the implied constant covers the remaining $H$. For this last assertion, coefficient bounds on $|z| = 2$ still give uniformity in $m$, and the same tail estimate is uniform in $j$.
\end{proof}

\subsection{The signed residuals}
\label{subsec:pointwise-singular-series-residuals}

We now translate the estimate for $\mathcal{A} - Q$ into bounds for the residuals defined in Section~\ref{sec:setup-and-bookkeeping}. Recall that
\begin{equation*}
G(k; H) = \frac{H^{k}}{k!}\left(1 - \binom{k}{2}\eta(H)\right), \qquad \Delta_{\mathrm{MS}}(k; H) = S(k; H) - G(k; H),
\end{equation*}
and that the signed truncation residual is
\begin{equation*}
\Delta_{\mathrm{tr}}(m, j; u, H) = -\frac{u^{m}}{m!}\sum_{\ell \, > \, j}\frac{(-u)^{\ell}}{\ell!}\left(1 - \binom{m + \ell}{2}\eta(H)\right).
\end{equation*}
Replacing $S$ by $G + \Delta_{\mathrm{MS}}$ in $\mathcal{A}$ expresses $\mathcal{A} - Q$ as a weighted sum of $\Delta_{\mathrm{MS}}$ plus $\Delta_{\mathrm{tr}}$. The theorem therefore controls their combination. We first bound $\Delta_{\mathrm{tr}}$ separately by an elementary factorial estimate, then deduce the bound for the weighted singular-series residual.

\begin{lemma}
\label{lem:pointwise-truncation-bound}
Put $K := m + j + 1$. For $H \ge 2$, $u \ge 0$ and integers $m, j \ge 0$,
\begin{equation}
\label{eq:pointwise-truncation-bound}
|\Delta_{\mathrm{tr}}(m, j; u, H)| \le \frac{e^{u}u^{K}}{m!(j + 1)!}\left\{1 + \frac{\eta(H)}{2}\left((K + u)^{2} + u\right)\right\}.
\end{equation}
Consequently there is an absolute constant $C_{1} \ge 1$ such that, for every fixed $U > 0$, uniformly for $0 \le u \le U$,
\begin{equation}
\label{eq:pointwise-truncation-factorial-bound}
|\Delta_{\mathrm{tr}}(m, j; u, H)| \ll_{U} \frac{(C_{1}u)^{m + j + 1}}{m!(j + 1)!}.
\end{equation}
\end{lemma}

\begin{proof}
Take absolute values in the defining series, write $\ell = j + 1 + r$, and use $\binom{m + \ell}{2} \le (K + r)^{2}/2$ and $(j + 1 + r)! \ge (j + 1)!r!$. The remaining sums are
\begin{equation*}
\sum_{r \, = \, 0}^{\infty}\frac{u^{r}}{r!} = e^{u}, \qquad \sum_{r \, = \, 0}^{\infty}\frac{u^{r}}{r!}(K + r)^{2} = e^{u}\left((K + u)^{2} + u\right).
\end{equation*}
This proves \eqref{eq:pointwise-truncation-bound}, since $\eta(H) > 0$. For the last assertion, $\eta(H)$ is bounded for $H \ge 2$, and $1 + K^{2} \ll 4^{K}$ for $K \ge 1$. Absorbing the resulting factor into the base of the power proves \eqref{eq:pointwise-truncation-factorial-bound}, with $C_{1} = 4$, for example.
\end{proof}

\begin{corollary}
\label{cor:weighted-ms-residual}
After increasing the absolute constant $C$ in Theorem~\ref{thm:weighted-singular-series-asymptotic} if necessary, uniformly in the same ranges,
\begin{equation}
\label{eq:weighted-ms-residual}
\sum_{\ell \, = \, 0}^{j}(-1)^{\ell}\binom{m + \ell}{m}\left(\frac{u}{H}\right)^{m + \ell}\Delta_{\mathrm{MS}}(m + \ell; H) \ll_{U, \epsilon} 2^{-m}H^{-3/2 + \epsilon} + \frac{(Cu)^{m + j + 1}}{m!(j + 1)!}.
\end{equation}
\end{corollary}

\begin{proof}
The definitions and the truncated exponential identity from the Bonferroni sandwich give the exact relation
\begin{equation}
\label{eq:pointwise-residual-decomposition}
\mathcal{A}(m, j; u, H) - Q(m; u, H) = \sum_{\ell \, = \, 0}^{j}(-1)^{\ell}\binom{m + \ell}{m}\left(\frac{u}{H}\right)^{m + \ell}\Delta_{\mathrm{MS}}(m + \ell; H) + \Delta_{\mathrm{tr}}(m, j; u, H).
\end{equation}
Apply \eqref{eq:weighted-singular-series-asymptotic} and \eqref{eq:pointwise-truncation-factorial-bound}, taking $C \ge \max\{C_{0}, C_{1}\}$.
\end{proof}

The identity \eqref{eq:pointwise-residual-decomposition} is useful because it preserves cancellation. The left side of \eqref{eq:weighted-ms-residual} is a signed sum. We have not estimated it by adding the absolute values of fixed-$k$ Montgomery--Soundararajan error bounds. Instead, the finite sieve represents the whole sum, its second and fourth moments control the expansion, and the uniform upper bound controls the tail. This distinction is what permits $j$ to grow.

\subsection{Integration over the starting points}
\label{subsec:integrated-singular-series-residuals}

We finish by connecting the pointwise estimates to the residuals in the sandwich. The result is uniform in the translation $M$ and includes $M = 0$. The restriction on $H$ below is only an upper bound; no lower bound on $H/\log N$ is needed here.

\begin{proposition}
\label{prop:integrated-singular-series-residuals}
Fix an integer $m \ge 0$ and real numbers $U, \epsilon > 0$. There is an absolute constant $C_{2} \ge 1$ such that, uniformly for integers $M \ge 0$, $N \ge 2$, $2 \le H \le U\log N$ and $j \ge 0$,
\begin{align}
\label{eq:integrated-singular-series-residual-bound}
\begin{split}
& |R_{\mathrm{MS}}(m, j; H, M, N)| + |R_{\mathrm{tr}}(m, j; H, M, N)|  \\
& \qquad \qquad \ll_{m, U, \epsilon} 
H^{-3/2 + \epsilon} + \frac{(C_{2}U)^{m + j + 1}}{m!(j + 1)!} + \frac{1 + (j + 1)(C_{2}H)^{m + j}}{\sqrt{N}}.
\end{split}
\end{align}
The implied constant is independent of $M$ and $j$.
\end{proposition}

\begin{corollary}
\label{cor:integrated-growing-truncation}
Fix $m \ge 0$ and $U, \delta > 0$. Suppose that $H \to \infty$, $H \le U\log N$, and let $L$ be the least even integer greater than or equal to $(1 + \delta)\log H/\log\log H$. Uniformly for $M \ge 0$ and $j \in \{L, L + 1\}$,
\begin{equation*}
|R_{\mathrm{MS}}(m, j; H, M, N)| + |R_{\mathrm{tr}}(m, j; H, M, N)| \ll_{m, U, \delta, \epsilon} H^{-3/2 + \epsilon} + H^{-1 - \delta + o(1)}.
\end{equation*}
In particular, both residuals are $o(1/H)$ if $0 < \epsilon < 1/2$.
\end{corollary}

\begin{proof}[Proof of Corollary~\ref{cor:integrated-growing-truncation}]
The factorial term in \eqref{eq:integrated-singular-series-residual-bound} is $H^{-1 - \delta + o(1)}$, as before. The logarithm of its final numerator is $O_{m, \delta}((\log H)^{2}/\log\log H) = o(H)$. Since $\log N \ge H/U$, the final term is bounded by $\exp(-H/(2U) + o(H))$, which is smaller than every fixed negative power of $H$. Apply Proposition~\ref{prop:integrated-singular-series-residuals}.
\end{proof}

\begin{proof}[Proof of Proposition~\ref{prop:integrated-singular-series-residuals}]
Write $a_{0} := \max\{2, M\}$. By the definition of $J$, the residual $R_{\mathrm{MS}}$ is the integral, divided by $N$, of the left side of \eqref{eq:weighted-ms-residual}, with $u = H/\log t$, over $a_{0} \le t \le M + N$. The corresponding statement for $R_{\mathrm{tr}}$ is its definition. Split these integrals at $t = \sqrt{N}$, omitting either piece if it is empty.

For $t \ge \sqrt{N}$, we have $u \le 2U$. The interval of integration has length at most $N$, so \eqref{eq:weighted-ms-residual} and \eqref{eq:pointwise-truncation-factorial-bound} give the first two terms in \eqref{eq:integrated-singular-series-residual-bound}, with $C_{2} \ge 2C$.

For $2 \le t < \sqrt{N}$, we use finite sums, rather than an absolute bound for an exponential series at the potentially large parameter $u$. There is an absolute constant $C_{3} \ge 1$ such that, for every $k \ge 0$,
\begin{equation*}
S(k; H) + |G(k; H)| \ll \frac{(C_{3}H)^{k}}{k!}.
\end{equation*}
For $S$ this follows from \eqref{eq:uniform-singular-series-upper-bound}. For $G$ it follows from its definition, the boundedness of $\eta(H)$, and $1 + k^{2} \ll 4^{k}$. Since $u \le H/\log 2$, taking absolute values in the finite weighted sum for $\Delta_{\mathrm{MS}} = S - G$ gives
\begin{equation*}
\left|\sum_{\ell \, = \, 0}^{j}(-1)^{\ell}\binom{m + \ell}{m}\left(\frac{u}{H}\right)^{m + \ell}\Delta_{\mathrm{MS}}(m + \ell; H)\right| \ll (j + 1)(C_{2}H)^{m + j},
\end{equation*}
provided $C_{2} \ge C_{3}/\log 2$ and $C_{2} \ge 1$.

The truncated exponential identity also writes $\Delta_{\mathrm{tr}}$ as the finite weighted sum involving $G$, minus $Q(m; u, H)$. The same estimate bounds that finite sum. Moreover, for fixed $m$, the expression $e^{-u}u^{m}(1 + (m + u)^{2})$ is bounded on $u \ge 0$, so $|Q(m; u, H)| \ll_{m} 1$. Thus
\begin{equation*}
|\Delta_{\mathrm{tr}}(m, j; u, H)| \ll_{m} 1 + (j + 1)(C_{2}H)^{m + j}.
\end{equation*}
The small-$t$ part has length at most $\sqrt{N}$. Dividing its contribution by $N$ gives the last term in \eqref{eq:integrated-singular-series-residual-bound} and completes the proof.
\end{proof}

Thus the Montgomery--Soundararajan and truncation residuals can both be made $o(1/H)$ unconditionally in the range relevant to the sandwich. The Hardy--Littlewood residual and the replacement of the integrated main term by $Q(m; \mu, H)$ are separate questions.

\section{Averaging the logarithmic density}
\label{sec:density-averaging}

The Bonferroni sandwich (Proposition \ref{prop:inclusion-exclusion-sandwich}) compares $P(m; H, M, N)$, the proportion of intervals containing exactly $m$ primes, with the averaged expression $F(m; H, M, N)$. We now replace $F(m; H, M, N)$ by $Q(m; \mu, H)$, evaluating the corrected Poisson expression at the averaged density. The error is smaller than $1/H$ in the range needed below, uniformly in $m$.

Recall from Section~\ref{sec:setup-and-bookkeeping} that
\begin{equation*}
Q(m; u, H) = \frac{e^{-u}u^{m}}{m!}\left\{1 - \frac{\eta(H)}{2}\left((m - u)^{2} - m\right)\right\}, \qquad \eta(H) = \frac{\log H + \log(2\pi) + \gamma - 1}{H},
\end{equation*}
with values at $u = 0$ understood by continuity. Its average over the varying logarithmic density is
\begin{equation*}
F(m; H, M, N) = \frac{1}{N}\int_{\max\{2, M\}}^{M + N}Q\left(m; \frac{H}{\log t}, H\right)dt.
\end{equation*}
The parameter at which we shall evaluate $Q$ is
\begin{equation*}
\mu = \mu(H, M, N) = \frac{1}{N}\int_{\max\{2, M\}}^{M + N}\frac{H}{\log t}dt,
\end{equation*}
and the signed residual to be estimated is
\begin{equation*}
\Delta_{\mathrm{av}}(m; H, M, N) = F(m; H, M, N) - Q(m; \mu, H).
\end{equation*}
The exact choice of $\mu$ matters: Taylor expansion of $Q(m; u, H)$ about $u = \mu$ cancels the integrated linear term, apart from a small correction when $M = 0$ or $1$. The remaining error is controlled by the mean square variation of $H/\log t$ from $\mu$.

\begin{proposition}[Density averaging]
\label{prop:density-averaging}
Fix real numbers $C, U > 0$. Uniformly for integers $N \ge 2$, $0 \le M \le CN$, $1 \le H \le U\log N$ and $m \ge 0$,
\begin{equation}
\label{eq:density-averaging-bound}
\left|\Delta_{\mathrm{av}}(m; H, M, N)\right| \ll_{C, U} \frac{H^{2}}{(\log N)^{4}} + \frac{1}{N}.
\end{equation}
The implied constant is independent of $m$.
\end{proposition}

Before proving the proposition, we record its consequence at the precision needed for the Bonferroni sandwich.

\begin{corollary}
\label{cor:density-averaging-negligible}
As $N \to \infty$, uniformly over the ranges in Proposition~\ref{prop:density-averaging},
\begin{equation*}
\Delta_{\mathrm{av}}(m; H, M, N) = o(1/H).
\end{equation*}
In particular, when $H \asymp \log N$, the residual is $O_{C, U}((\log N)^{-2})$, uniformly for $m \ge 0$.
\end{corollary}

\begin{proof}[Proof of Corollary~\ref{cor:density-averaging-negligible}]
Multiplying the right side of \eqref{eq:density-averaging-bound} by $H$ gives
\begin{equation*}
\frac{H^{3}}{(\log N)^{4}} + \frac{H}{N} \le \frac{U^{3}}{\log N} + \frac{U\log N}{N} \to 0.
\end{equation*}
The second assertion follows from $H \le U\log N$.
\end{proof}

The uniformity in $m$ concerns an absolute error. When $m$ is large, the Poisson main term and the arithmetic correction in $Q(m; \mu, H)$ may both be smaller in magnitude than this error bound. An estimate that distinguishes these terms requires further restrictions on $m$.

The proof uses Taylor expansion of $Q(m; u, H)$ about $u = \mu$, with the remainder controlled by its second derivative. We give the derivative and integral estimates explicitly; no probabilistic limit theorem is needed.

\subsection{Uniform bounds for the corrected Poisson expression}
\label{subsec:averaging-uniform-derivatives}

For integers $m \ge 0$ and real $u \ge 0$, write
\begin{equation*}
p_{m}(u) := \frac{e^{-u}u^{m}}{m!},
\end{equation*}
again using continuity at $u = 0$, and put $p_{m}(u) = 0$ for negative integers $m$. These functions are the Poisson probabilities. Their derivatives are finite differences of the same functions, which allows us to avoid bounds involving powers of $m$. All derivatives below are with respect to $u$.

\begin{lemma}
\label{lem:averaging-uniform-derivatives}
For all integers $m, r \ge 0$ and real $u \ge 0$,
\begin{equation}
\label{eq:averaging-poisson-derivatives}
p_{m}^{(r)}(u) = \sum_{i \, = \, 0}^{r}(-1)^{r - i}\binom{r}{i}p_{m - i}(u), \qquad \left|p_{m}^{(r)}(u)\right| \le 2^{r}.
\end{equation}
For $H \ge 1$, the corrected Poisson approximation satisfies
\begin{equation}
\label{eq:averaging-q-derivative-form}
Q(m; u, H) = p_{m}(u) - \frac{\eta(H)}{2}u^{2}p_{m}''(u)
\end{equation}
and
\begin{equation}
\label{eq:averaging-q-derivative-bound}
\left|\frac{\partial^{r}}{\partial u^{r}}Q(m; u, H)\right| \le 2^{r}\left\{1 + \frac{|\eta(H)|}{2}\left(4u^{2} + 4ru + r(r - 1)\right)\right\}.
\end{equation}
Consequently, for fixed $V > 0$ and integer $r \ge 0$, these derivatives are $O_{V, r}(1)$ uniformly for $0 \le u \le V$, $H \ge 1$ and $m \ge 0$.
\end{lemma}

\begin{proof}
Differentiation gives $p_{m}' = p_{m - 1} - p_{m}$, including when $m = 0$. Induction on $r$ proves the identity in \eqref{eq:averaging-poisson-derivatives}. The exponential series gives
\begin{equation*}
\sum_{m \, = \, 0}^{\infty}p_{m}(u) = 1,
\end{equation*}
so $0 \le p_{m}(u) \le 1$ for every integer $m$. Taking absolute values in the finite-difference identity proves the derivative bound.

For $u > 0$, direct differentiation also gives
\begin{equation*}
u^{2}p_{m}''(u) = \left(m(m - 1) - 2mu + u^{2}\right)p_{m}(u) = \left((m - u)^{2} - m\right)p_{m}(u).
\end{equation*}
Substitution into the definition of $Q$ proves \eqref{eq:averaging-q-derivative-form}, and continuity extends it to $u = 0$. For $r \ge 2$, Leibniz's rule gives
\begin{equation*}
\frac{d^{r}}{du^{r}}\left(u^{2}p_{m}''(u)\right) = u^{2}p_{m}^{(r + 2)}(u) + 2ru\,p_{m}^{(r + 1)}(u) + r(r - 1)p_{m}^{(r)}(u).
\end{equation*}
For $r = 0$ or $1$, the same formula holds with the terms having zero coefficients omitted. Applying \eqref{eq:averaging-poisson-derivatives} to each term proves \eqref{eq:averaging-q-derivative-bound}. Finally, $\eta(H)$ is bounded for $H \ge 1$, which proves the assertion on bounded $u$-intervals.
\end{proof}

We shall use the lemma with $r = 0, 1, 2$. In particular, the case $r = 0$ gives a bound valid even when $u$ is large:
\begin{equation}
\label{eq:averaging-q-global-bound}
|Q(m; u, H)| \le 1 + 2|\eta(H)|u^{2} \ll 1 + u^{2}.
\end{equation}
The implied constant is absolute. This crude bound will suffice on the short portion of the integral where $H/\log t$ need not be bounded.

\subsection{The averaging residual}
\label{subsec:averaging-residual-estimate}

We first estimate the variation of the logarithmic density. In the rest of this subsection write
\begin{equation*}
I := [\max\{2, M\}, M + N], \qquad u(t) := \frac{H}{\log t}, \qquad a := \frac{H}{\log N}, \qquad d := \frac{\max\{2 - M, 0\}}{N}.
\end{equation*}
Thus the interval $I$ has length $N(1 - d)$. We retain the normalization by $N$ used in Section \ref{sec:setup-and-bookkeeping} (bookkeeping). In particular,
\begin{equation}
\label{eq:averaging-mass-identities}
\frac{1}{N}\int_{I}dt = 1 - d, \qquad \frac{1}{N}\int_{I}u(t)dt = \mu, \qquad \frac{1}{N}\int_{I}(u(t) - \mu)dt = d\mu.
\end{equation}
The last integral vanishes for $M \ge 2$, but not necessarily for $M = 0$ or $1$.

\begin{lemma}
\label{lem:logarithmic-density-variation}
Fix $C > 0$. For integers $N \ge 2$, $0 \le M \le CN$ and $H \ge 1$,
\begin{equation}
\label{eq:averaging-density-comparison}
\frac{1}{N}\int_{I}(u(t) - a)^{2}dt \ll_{C} \frac{H^{2}}{(\log N)^{4}},
\end{equation}
and
\begin{equation}
\label{eq:averaging-mean-comparison}
|\mu - a| \ll_{C} \frac{H}{(\log N)^{2}}, \qquad 0 \le \mu \ll_{C} \frac{H}{\log N}.
\end{equation}
Consequently,
\begin{equation}
\label{eq:averaging-centered-square}
\frac{1}{N}\int_{I}(u(t) - \mu)^{2}dt \ll_{C} \frac{H^{2}}{(\log N)^{4}}.
\end{equation}
No upper bound on $H$ is needed in this lemma.
\end{lemma}

\begin{proof}
We may suppose $N \ge 4$. For $N = 2, 3$, both $u(t)$ and $a$ are $O(H)$ throughout $I$, and the assertions follow by increasing the implied constants.

Split $I$ into its intersections with $[2, \sqrt N)$ and $[\sqrt N, \infty)$. On the first part, $u(t)$ and $a$ are $O(H)$, and its length is at most $\sqrt N$. Its contribution to the left side of \eqref{eq:averaging-density-comparison} is therefore $O(H^{2}/\sqrt N)$.

On the second part,
\begin{equation*}
|u(t) - a| = \frac{H|\log(t/N)|}{\log t\,\log N} \le \frac{2H|\log(t/N)|}{(\log N)^{2}}.
\end{equation*}
Since $M + N \le (C + 1)N$, the change of variable $s = t/N$ gives
\begin{equation*}
\frac{1}{N}\int_{I\,\cap\,[\sqrt N,\,\infty)}|\log(t/N)|^{2}dt \le \int_{0}^{C + 1}|\log s|^{2}ds \ll_{C} 1.
\end{equation*}
The integral at zero converges; for example, an antiderivative on $s > 0$ is $s((\log s)^{2} - 2\log s + 2)$, which tends to zero as $s \to 0$. Combining the two parts yields
\begin{equation*}
\frac{1}{N}\int_{I}(u(t) - a)^{2}dt \ll_{C} \frac{H^{2}}{\sqrt N} + \frac{H^{2}}{(\log N)^{4}} \ll_{C} \frac{H^{2}}{(\log N)^{4}}.
\end{equation*}
Here the last step uses the boundedness of $(\log N)^{4}/\sqrt N$ for $N \ge 4$.

To compare $\mu$ with $a$, note that
\begin{equation*}
\mu - a = \frac{1}{N}\int_{I}(u(t) - a)dt - da.
\end{equation*}
Cauchy--Schwarz, the fact that the length of $I$ is at most $N$, and \eqref{eq:averaging-density-comparison} give
\begin{equation*}
|\mu - a| \le \left(\frac{1}{N}\int_{I}(u(t) - a)^{2}dt\right)^{1/2} + da \ll_{C} \frac{H}{(\log N)^{2}} + \frac{H}{N\log N} \ll_{C} \frac{H}{(\log N)^{2}}.
\end{equation*}
This proves \eqref{eq:averaging-mean-comparison}. Finally, the pointwise inequality
\begin{equation*}
(u(t) - \mu)^{2} \le 2(u(t) - a)^{2} + 2(a - \mu)^{2}
\end{equation*}
proves \eqref{eq:averaging-centered-square} after integration.
\end{proof}

We can now compare the average of $Q$ with its value at $\mu$. For $M \ge 2$, the calculation is the familiar cancellation of the linear term in a Taylor expansion about the mean. We keep the small defect in the total mass throughout, so the same proof covers $M = 0$ and $1$.

\begin{proof}[Proof of Proposition~\ref{prop:density-averaging}]
Within this proof write $q(u) := Q(m; u, H)$. By \eqref{eq:averaging-mean-comparison} and $H \le U\log N$, we have $\mu \ll_{C} U$. Lemma~\ref{lem:averaging-uniform-derivatives} therefore gives
\begin{equation*}
|q(\mu)| + |q'(\mu)| \ll_{C, U} 1,
\end{equation*}
uniformly in $m$. We may suppose $N \ge 4$: when $N = 2, 3$, the upper bound on $H$ bounds both $u(t)$ and $\mu$ in terms of $U$, so the proposition follows from the same lemma by increasing the implied constant.

Put
\begin{equation*}
g(u) := q(u) - q(\mu) - (u - \mu)q'(\mu).
\end{equation*}
For $t \ge \sqrt N$, we have $0 \le u(t) \le 2U$. The entire segment between $u(t)$ and $\mu$ is therefore contained in a bounded interval depending only on $C$ and $U$. Taylor's theorem with integral remainder gives
\begin{equation*}
g(u) = (u - \mu)^{2}\int_{0}^{1}(1 - s)q''(\mu + s(u - \mu))ds.
\end{equation*}
The uniform bound for $q''$ from Lemma~\ref{lem:averaging-uniform-derivatives}, followed by \eqref{eq:averaging-centered-square}, now gives
\begin{equation*}
\frac{1}{N}\int_{I\,\cap\,[\sqrt N,\,\infty)}|g(u(t))|dt \ll_{C, U} \frac{1}{N}\int_{I}(u(t) - \mu)^{2}dt \ll_{C, U} \frac{H^{2}}{(\log N)^{4}}.
\end{equation*}

For $2 \le t < \sqrt N$, use \eqref{eq:averaging-q-global-bound} to bound $q(u(t))$, and \eqref{eq:averaging-q-derivative-bound} with $r = 0, 1$ to bound $q(\mu)$ and $q'(\mu)$. Since $u(t) \le H/\log 2$ and $\mu \ll_{C} U$,
\begin{equation*}
|g(u(t))| \le |q(u(t))| + |q(\mu)| + |u(t) - \mu|\,|q'(\mu)| \ll_{C, U} H^{2}.
\end{equation*}
This part has length at most $\sqrt N$, so
\begin{equation*}
\frac{1}{N}\int_{I\,\cap\,[2,\,\sqrt N)}|g(u(t))|dt \ll_{C, U} \frac{H^{2}}{\sqrt N} \ll_{C, U} \frac{H^{2}}{(\log N)^{4}}.
\end{equation*}
Both estimates are independent of $m$.

It remains to integrate the constant and linear terms exactly. The identities \eqref{eq:averaging-mass-identities} give
\begin{equation}
\label{eq:averaging-endpoint-correction}
\Delta_{\mathrm{av}}(m; H, M, N) = d\left(\mu q'(\mu) - q(\mu)\right) + \frac{1}{N}\int_{I}g(u(t))dt.
\end{equation}
The first term in \eqref{eq:averaging-endpoint-correction} is $O_{C, U}(1/N)$ because $0 \le d \le 2/N$. The bounds for the remaining integral prove \eqref{eq:density-averaging-bound}.
\end{proof}

To bound $F(m; H, M, N) - Q(m; \mu, H)$, a direct application of the mean value theorem to $Q$ as a function of $u$ would bound the difference by the mean absolute deviation of $H/\log t$ from $\mu$. Our estimates would then give only $O_{C, U}(H/(\log N)^{2} + 1/N)$. By retaining the linear term in the Taylor expansion about $u = \mu$ and integrating it before taking absolute values, we exploit its cancellation and obtain $O_{C, U}(H^{2}/(\log N)^{4} + 1/N)$ instead. When $H \asymp \log N$, this improves the error bound from $O(1/H)$ to $O(1/H^{2})$.

\subsection{The effective mean}
\label{subsec:averaging-effective-mean}

The preceding proof retained the logarithmic integral in $\mu$ exactly. To relate it to a single logarithm, we now expand that integral to two terms. The second term depends on the position of the averaging interval. This dependence accounts for the different normalizations in $(0, N]$ and $(N, 2N]$.

For $\alpha \ge 0$, put
\begin{equation}
\label{eq:averaging-logarithmic-shift}
c(\alpha) := \int_{\alpha}^{\alpha + 1}\log s\,ds = (\alpha + 1)\log(\alpha + 1) - \alpha\log\alpha - 1,
\end{equation}
where $0\log 0$ is interpreted as zero. The integral at $\alpha = 0$ is improper and converges. In particular, $c$ is continuous on $[0, \infty)$ and bounded on each interval $[0, C]$.

\begin{proposition}
\label{prop:averaging-effective-mean}
Fix $C > 0$, and put $\alpha = M/N$. As $N \to \infty$, uniformly for integers $0 \le M \le CN$ and $H \ge 1$,
\begin{equation}
\label{eq:averaging-mean-two-terms}
\mu(H, M, N) = \frac{H}{\log N} - c(\alpha)\frac{H}{(\log N)^{2}} + O_{C}\left(\frac{H}{(\log N)^{3}}\right).
\end{equation}
Equivalently, for all sufficiently large $N$,
\begin{equation}
\label{eq:averaging-mean-shifted-logarithm}
\mu(H, M, N) = \frac{H}{\log N + c(\alpha)} + O_{C}\left(\frac{H}{(\log N)^{3}}\right).
\end{equation}
The implied constants are independent of $H$ and $M$; in particular, no upper bound on $H$ is needed for these expansions.
\end{proposition}

\begin{proof}
Write $b = \log N$ within this proof. For $t > 1$, the identity
\begin{equation}
\label{eq:averaging-reciprocal-logarithm}
\frac{1}{\log t} = \frac{1}{b} - \frac{\log(t/N)}{b^{2}} + \frac{(\log(t/N))^{2}}{b^{2}\log t}
\end{equation}
is exact. On $t \ge \sqrt N$, the absolute value of the last term is at most $2|\log(t/N)|^{2}/b^{3}$. As in the proof of Lemma~\ref{lem:logarithmic-density-variation}, its integral over the relevant range, divided by $N$, is $O_{C}(b^{-3})$.

We compare the original integral with the first two terms of \eqref{eq:averaging-reciprocal-logarithm} integrated over the full interval $[M, M + N]$. Only the portion below $\sqrt N$ needs separate treatment, including the missing interval below $2$ when $M = 0$ or $1$. The original integrand contributes at most
\begin{equation*}
\frac{1}{N}\int_{2}^{\sqrt N}\frac{dt}{\log t} \le \frac{1}{\sqrt N\log 2}.
\end{equation*}
For the two elementary terms we use
\begin{equation*}
\frac{1}{N}\int_{0}^{\sqrt N}|\log(t/N)|dt = \frac{b/2 + 1}{\sqrt N}.
\end{equation*}
Including the integral of $1/b$, their contribution below $\sqrt N$ is $O(N^{-1/2}/b)$ for $N \ge 4$. Thus the entire error made on this portion is $O(N^{-1/2})$, which is $O(b^{-3})$.

We have proved
\begin{equation*}
\frac{\mu}{H} = \frac{1}{N}\int_{M}^{M + N}\left(\frac{1}{b} - \frac{\log(t/N)}{b^{2}}\right)dt + O_{C}(b^{-3}).
\end{equation*}
The first term integrates to $1/b$. In the second, the change of variable $s = t/N$ gives $c(\alpha)$ as defined in \eqref{eq:averaging-logarithmic-shift}. This proves \eqref{eq:averaging-mean-two-terms}, including uniformly when $\alpha$ approaches zero.

Since $c(\alpha)$ is bounded for $0 \le \alpha \le C$, the denominator $b + c(\alpha)$ is positive for all sufficiently large $N$, uniformly in $\alpha$. The elementary identity
\begin{equation*}
\frac{1}{b + c(\alpha)} = \frac{1}{b} - \frac{c(\alpha)}{b^{2}} + \frac{c(\alpha)^{2}}{b^{2}(b + c(\alpha))}
\end{equation*}
then proves \eqref{eq:averaging-mean-shifted-logarithm}.
\end{proof}

For the two averaging intervals of particular interest, 
\begin{equation*}
c(0) = -1, \qquad c(1) = 2\log 2 - 1.
\end{equation*}
Consequently, \eqref{eq:averaging-mean-shifted-logarithm} gives
\begin{equation*}
\mu(H, 0, N) = \frac{H}{\log N - 1} + O\left(\frac{H}{(\log N)^{3}}\right)
\end{equation*}
and
\begin{equation*}
\mu(H, N, N) = \frac{H}{\log N + 2\log 2 - 1} + O\left(\frac{H}{(\log N)^{3}}\right).
\end{equation*}
These shifted logarithms describe the average of the varying density $1/\log t$ over the chosen interval. They do not replace that density pointwise.

\begin{corollary}
\label{cor:averaging-shifted-parameter}
Fix $C, U > 0$, and write
\begin{equation*}
\lambda := \frac{H}{\log N + c(M/N)},
\end{equation*}
where the function $c$ is defined in \eqref{eq:averaging-logarithmic-shift}. As $N \to \infty$, uniformly for integers $0 \le M \le CN$, $1 \le H \le U\log N$ and $m \ge 0$,
\begin{equation*}
F(m; H, M, N) = Q(m; \lambda, H) + O_{C, U}\left(\frac{H}{(\log N)^{3}} + \frac{1}{N}\right).
\end{equation*}
The error is $o(1/H)$ uniformly in these ranges.
\end{corollary}

\begin{proof}
Proposition~\ref{prop:averaging-effective-mean} gives $|\mu - \lambda| \ll_{C} H/(\log N)^{3}$. Both $\mu$ and $\lambda$ are nonnegative and bounded in terms of $C$ and $U$ for all sufficiently large $N$. The first-derivative bound in Lemma~\ref{lem:averaging-uniform-derivatives} and the mean value theorem therefore give
\begin{equation*}
|Q(m; \mu, H) - Q(m; \lambda, H)| \ll_{C, U} \frac{H}{(\log N)^{3}},
\end{equation*}
uniformly in $m$. Combine this with Proposition~\ref{prop:density-averaging}, using
\begin{equation*}
\frac{H^{2}}{(\log N)^{4}} \le \frac{UH}{(\log N)^{3}}.
\end{equation*}
Finally, multiplying the error bound in the corollary by $H$ gives
\begin{equation*}
H\left(\frac{H}{(\log N)^{3}} + \frac{1}{N}\right) = \frac{H^{2}}{(\log N)^{3}} + \frac{H}{N} \le \frac{U^{2}}{\log N} + \frac{U\log N}{N} \to 0.
\end{equation*}
Thus the error is $o(1/H)$ uniformly in the stated ranges.
\end{proof}

The second term of \eqref{eq:averaging-mean-two-terms} explains why replacing $\mu$ by $H/\log N$ would lose the required precision. When $H \asymp \log N$ and $c(M/N)$ stays away from zero, the change in the parameter has order $1/H$. Its first-order effect on $Q$ can therefore affect the part of the singular-series correction proportional to $1/H$. Keeping $\mu$ exactly, or using the shifted logarithm with its stated error, makes the averaging error $o(1/H)$.

\section{From prime tuples to prime counts}
\label{sec:prime-tuples-to-prime-counts}

The preceding sections leave one residual to control: the signed Hardy--Littlewood residual in the Bonferroni sandwich (Proposition \ref{prop:inclusion-exclusion-sandwich}). We first give a quantitative transfer statement, keeping this residual explicit. We then derive sufficient conditions on averaged prime-tuple counts and on individual prime-tuple counts. These conditions are hypotheses, not consequences of the usual conjecture for each fixed tuple. The distinction matters because both the set of shifts and its cardinality vary with $N$.

Except where the full notation is useful, we suppress $H$, $M$ and $N$ as arguments, as in Section~\ref{sec:bonferroni-sandwich}. The integer $m$ is fixed in the asymptotic conclusions of this section. Some of the intermediate inequalities hold uniformly in $m$, but the integrated singular-series estimate used here was stated for fixed $m$.

Recall that the averaged prime-tuple residual and its weighted sum are
\begin{equation*}
\Delta_{\mathrm{HL}}(k) = T(k) - J(k)S(k)
\end{equation*}
and
\begin{equation*}
R_{\mathrm{HL}}(m, j) = \sum_{\ell \, = \, 0}^{j}(-1)^{\ell}\binom{m + \ell}{m}\Delta_{\mathrm{HL}}(m + \ell).
\end{equation*}
Here $T(k)$ is the average number of ways to choose $k$ primes from an interval, and $J(k)S(k)$ is its Hardy--Littlewood prediction. It is the signed sum $R_{\mathrm{HL}}$, at the two truncation orders in the Bonferroni sandwich, that the conditional argument needs to control. Bounds for the individual residuals $\Delta_{\mathrm{HL}}(k)$ will provide sufficient conditions.

\subsection{A quantitative transfer statement}
\label{subsec:quantitative-prime-count-transfer}

We collect the unconditional singular-series and truncation bounds in an error term $E_{\mathrm{an}}$, where the subscript stands for ``analytic''. Keeping the Hardy--Littlewood residual explicit shows what precision is needed and allows any suitable bound for that residual to be inserted directly, without revisiting the preceding proofs.

\begin{proposition}
\label{prop:quantitative-prime-count-transfer}
There is an absolute constant $D \ge 1$ with the following property. Fix an integer $m \ge 0$ and real numbers $U, \epsilon > 0$. Let $M \ge 0$, $N \ge 2$ and $2 \le H \le U\log N$ be integers, and let $L \ge 0$ be even. Write
\begin{equation}
\label{eq:transfer-analytic-error}
E_{\mathrm{an}} := H^{-3/2 + \epsilon} + \frac{(DU)^{m + L + 1}}{m!(L + 1)!} + \frac{1 + (L + 2)(DH)^{m + L + 1}}{\sqrt{N}}.
\end{equation}
Then, uniformly in $M$ and $L$,
\begin{equation}
\label{eq:transfer-integrated-bound}
P(m) = F(m) + O_{m, U, \epsilon}\left(\max_{j \, \in \, \{L, L + 1\}}|R_{\mathrm{HL}}(m, j)| + E_{\mathrm{an}}\right).
\end{equation}
If also $0 \le M \le CN$ for a fixed $C > 0$, then
\begin{equation}
\label{eq:transfer-averaged-parameter-bound}
P(m) = Q(m; \mu) + O_{m, C, U, \epsilon}\left(\max_{j \, \in \, \{L, L + 1\}}|R_{\mathrm{HL}}(m, j)| + E_{\mathrm{an}} + \frac{H^{2}}{(\log N)^{4}} + \frac{1}{N}\right).
\end{equation}
Both statements are unconditional.
\end{proposition}

\begin{proof}
Proposition \ref{prop:inclusion-exclusion-sandwich} (Bonferroni sandwich) gives
\begin{equation*}
|P(m) - F(m)| \le \max_{j \, \in \, \{L, L + 1\}}\left\{|R_{\mathrm{HL}}(m, j)| + |R_{\mathrm{MS}}(m, j)| + |R_{\mathrm{tr}}(m, j)|\right\}.
\end{equation*}
Apply Proposition~\ref{prop:integrated-singular-series-residuals} at $j = L$ and $j = L + 1$. Choose $D$ at least as large as the absolute constant $C_{2}$ in that proposition. The last term there is bounded by the last term in \eqref{eq:transfer-analytic-error}. For the factorial term, increasing $j$ from $L$ to $L + 1$ multiplies the expression with base $DU$ by $DU/(L + 2)$, which is bounded in terms of $U$. This proves \eqref{eq:transfer-integrated-bound}. The identity $F(m) = Q(m; \mu) + \Delta_{\mathrm{av}}(m)$ and Proposition~\ref{prop:density-averaging} give \eqref{eq:transfer-averaged-parameter-bound}.
\end{proof}

The truncation order can now be chosen once and for all. We record the resulting conditional statement before examining sufficient conditions for its hypothesis.

\begin{corollary}
\label{cor:weighted-hl-transfer}
Fix an integer $m \ge 0$ and real numbers $U, \delta > 0$. Suppose that $H \to \infty$ and $H \le U\log N$, and let $L$ be the least even integer satisfying
\begin{equation}
\label{eq:transfer-truncation-order}
L \ge (1 + \delta)\frac{\log H}{\log\log H}.
\end{equation}
On any family of integers $M \ge 0$, $N \ge 2$ and $H$ in these ranges for which
\begin{equation}
\label{eq:weighted-hl-sufficient-condition}
\max_{j \, \in \, \{L, L + 1\}}|R_{\mathrm{HL}}(m, j; H, M, N)| = o(1/H)
\end{equation}
holds uniformly, we have, with the same uniformity,
\begin{equation}
\label{eq:conditional-integrated-prime-count}
P(m; H, M, N) = F(m; H, M, N) + o(1/H).
\end{equation}
If also $0 \le M \le CN$ for a fixed $C > 0$, then
\begin{equation}
\label{eq:conditional-mean-prime-count}
P(m; H, M, N) = Q(m; \mu, H) + o(1/H).
\end{equation}
\end{corollary}

\begin{proof}
Fix $0 < \epsilon < 1/2$ in Proposition~\ref{prop:quantitative-prime-count-transfer}. Stirling's formula and the choice of $L$ give
\begin{equation*}
\log\left(\frac{(DU)^{m + L + 1}}{m!(L + 1)!}\right) = -(L + 1)\log(L + 1) + O_{m, U}(L + 1) = -(1 + \delta)\log H + o(\log H).
\end{equation*}
Thus this term is $H^{-1 - \delta + o(1)}$. Moreover,
\begin{equation*}
\log\left(1 + (L + 2)(DH)^{m + L + 1}\right) = O_{m, \delta}\left(\frac{(\log H)^{2}}{\log\log H}\right) = o(H).
\end{equation*}
Since $\log N \ge H/U$, the last term in \eqref{eq:transfer-analytic-error} is at most $\exp(-H/(2U) + o(H))$. All three terms in $E_{\mathrm{an}}$ are therefore $o(1/H)$. This proves \eqref{eq:conditional-integrated-prime-count}. For the second assertion, the remaining terms in \eqref{eq:transfer-averaged-parameter-bound} satisfy
\begin{equation*}
H\left(\frac{H^{2}}{(\log N)^{4}} + \frac{1}{N}\right) \le \frac{U^{3}}{\log N} + \frac{U\log N}{N} \to 0.
\end{equation*}
\end{proof}

Condition \eqref{eq:weighted-hl-sufficient-condition} retains cancellation within each tuple average and between the different tuple sizes. It is a sufficient condition tailored to the argument, rather than a claim of logical necessity. Indeed, the signed sandwich requires only an upper bound for $R_{\mathrm{HL}}(m, L)$ and a lower bound for $R_{\mathrm{HL}}(m, L + 1)$. The weaker condition
\begin{equation*}
\max\left\{0, R_{\mathrm{HL}}(m, L), -R_{\mathrm{HL}}(m, L + 1)\right\} = o(1/H)
\end{equation*}
would suffice as well. We use the symmetric form because it follows conveniently from bounds on prime-tuple residuals.

\subsection{Sufficient Hardy--Littlewood hypotheses}
\label{subsec:sufficient-hardy-littlewood-hypotheses}

We first remove the empty tuple and the singleton tuples from the unproved part of the argument. The identity for $k = 0$ is exact. For $k = 1$, the quantitative prime number theorem supplies more precision than we need.

\begin{lemma}
\label{lem:empty-and-singleton-hl-residuals}
We have
\begin{equation*}
\Delta_{\mathrm{HL}}(0; H, M, N) = \frac{\max\{2 - M, 0\}}{N}.
\end{equation*}
For every fixed $C, B > 0$, uniformly for integers $0 \le M \le CN$, $N \ge 2$ and $H \ge 1$,
\begin{equation}
\label{eq:transfer-singleton-residual}
\Delta_{\mathrm{HL}}(1; H, M, N) \ll_{C, B} \frac{H}{(\log N)^{B}} + \frac{H^{2}}{N}.
\end{equation}
In particular, for fixed $U > 0$, both residuals are $o(1/H)$ uniformly when $H \le U\log N$ and $N \to \infty$.
\end{lemma}

\begin{proof}
The empty-tuple identity follows from $T(0) = S(0) = 1$ and the length of the integral defining $J(0)$. For the singleton residual, the endpoint calculation in Subsection~\ref{subsec:corrected-poisson-and-averaging} gives
\begin{equation*}
T(1) = \frac{H}{N}\left(\pi(M + N) - \pi(M)\right) + O\left(\frac{H^{2}}{N}\right).
\end{equation*}
We use the quantitative prime number theorem in the form
\begin{equation}
\label{eq:transfer-quantitative-pnt}
\pi(x) = \int_{2}^{x}\frac{dt}{\log t} + O_{B}\left(\frac{x}{(\log x)^{B}}\right) \qquad (x \ge 2),
\end{equation}
valid for each fixed $B > 0$. This follows from the stronger exponential error bound in \cite[Theorem~6.9]{MV2006} and is the only additional analytic input in this section.

At the upper endpoint, \eqref{eq:transfer-quantitative-pnt} has error $O_{C, B}(N/(\log N)^{B})$. The same bound holds at the lower endpoint uniformly in $M$: use \eqref{eq:transfer-quantitative-pnt} when $M \ge \sqrt{N}$ and the trivial bound $O(\sqrt{N})$ for the prime count and the integral when $2 \le M < \sqrt{N}$. For $M = 0, 1$, the lower prime count and the lower integral are both absent. Consequently,
\begin{equation*}
\pi(M + N) - \pi(M) = \int_{\max\{2, M\}}^{M + N}\frac{dt}{\log t} + O_{C, B}\left(\frac{N}{(\log N)^{B}}\right).
\end{equation*}
Since $S(1; H) = H$, substitution proves \eqref{eq:transfer-singleton-residual}. To obtain the final assertion, take $B > 2$ and multiply the bound by $H$, using $H \le U\log N$.
\end{proof}

We next make the binomial weights explicit. In the following lemma, the numbers $e_{k}$ are bounds for the already averaged, signed residuals $\Delta_{\mathrm{HL}}(k)$; no absolute values are taken inside the sum over sets of shifts defining those residuals.

\begin{lemma}
\label{lem:weighted-hl-error-envelopes}
Let $m, L \ge 0$ be integers, and put $K := m + L + 1$. Suppose that $e_{k} \ge 0$ and $|\Delta_{\mathrm{HL}}(k)| \le e_{k}$ for $2 \le k \le K$. Then
\begin{equation}
\label{eq:weighted-hl-positive-envelope}
\max_{j \, \in \, \{L, L + 1\}}|R_{\mathrm{HL}}(m, j)| \le \frac{2}{N} + |\Delta_{\mathrm{HL}}(1)| + \sum_{k \, = \, \max\{2, m\}}^{K}\binom{k}{m}e_{k}.
\end{equation}
In particular, if $A \ge 1$, $\xi \ge 0$ and
\begin{equation}
\label{eq:averaged-hl-factorial-envelope}
|\Delta_{\mathrm{HL}}(k)| \le \frac{\xi A^{k}}{k!}\left(\frac{H}{\log N}\right)^{k} \qquad (2 \le k \le K),
\end{equation}
then
\begin{equation}
\label{eq:weighted-hl-factorial-envelope}
\max_{j \, \in \, \{L, L + 1\}}|R_{\mathrm{HL}}(m, j)| \le \frac{2}{N} + |\Delta_{\mathrm{HL}}(1)| + \frac{\xi(AH/\log N)^{m}}{m!}\exp\left(\frac{AH}{\log N}\right).
\end{equation}
These inequalities do not require $m$ or $L$ to be fixed.
\end{lemma}

\begin{proof}
In the definition of $R_{\mathrm{HL}}(m, j)$, put $k = m + \ell$ and take absolute values. The term $k = 0$, when present, is at most $2/N$. The coefficient of the term $k = 1$, when present, has absolute value $1$. The remaining terms are bounded by the sum in \eqref{eq:weighted-hl-positive-envelope}.

For the second assertion, put $a := AH/\log N$. The factorial in \eqref{eq:averaged-hl-factorial-envelope} cancels part of the binomial coefficient:
\begin{equation*}
\binom{k}{m}\frac{a^{k}}{k!} = \frac{a^{m}}{m!}\frac{a^{k - m}}{(k - m)!} \qquad (k \ge m).
\end{equation*}
Completing the resulting positive exponential series gives
\begin{equation*}
\sum_{k \, = \, \max\{2, m\}}^{K}\binom{k}{m}\frac{\xi a^{k}}{k!} \le \frac{\xi a^{m}}{m!}\sum_{\ell \, = \, 0}^{\infty}\frac{a^{\ell}}{\ell!} = \frac{\xi a^{m}}{m!}e^{a}.
\end{equation*}
This proves \eqref{eq:weighted-hl-factorial-envelope}.
\end{proof}

For fixed $A, U$, the factor multiplying $\xi$ in \eqref{eq:weighted-hl-factorial-envelope} is bounded independently of $L$. It is even bounded independently of $m$, since $a^{m}/m! \le e^{a}$ for $a \ge 0$. Thus the growing number of terms does not by itself entail a loss: the factorial decay absorbs the binomial weights.

More explicitly, retain the ranges of Corollary~\ref{cor:weighted-hl-transfer}, impose $0 \le M \le CN$, and put $K = m + L + 1$. A sufficient averaged Hardy--Littlewood hypothesis is that, for some fixed $A \ge 1$,
\begin{equation}
\label{eq:normalized-averaged-hl-condition}
\max_{2 \, \le \, k \, \le \, K}\frac{k!}{A^{k}}\left(\frac{\log N}{H}\right)^{k}|\Delta_{\mathrm{HL}}(k; H, M, N)| = o(1/H),
\end{equation}
uniformly over the stated ranges. Under this hypothesis, apply \eqref{eq:weighted-hl-factorial-envelope} with $\xi$ equal to the maximum in \eqref{eq:normalized-averaged-hl-condition}, and use Lemma~\ref{lem:empty-and-singleton-hl-residuals}. Thus \eqref{eq:normalized-averaged-hl-condition} implies \eqref{eq:weighted-hl-sufficient-condition}.

One can weaken \eqref{eq:normalized-averaged-hl-condition} by prescribing different bounds for different $k$ and requiring only that the sum in \eqref{eq:weighted-hl-positive-envelope} be $o(1/H)$. One can weaken it further by estimating $R_{\mathrm{HL}}$ with its signs intact. The bound \eqref{eq:averaged-hl-factorial-envelope} is convenient because its factorial decay absorbs the binomial weights, leaving only the requirement $\xi = o(1/H)$.

For comparison, we give two sufficient hypotheses on individual tuples. They sacrifice cancellation over sets of shifts but show how much stronger a power saving in $N$ would be than the precision required here.

For both hypotheses, fix an integer $m \ge 0$ and real numbers $C, U, \delta > 0$. Consider integers $H$, $M$ and $N$ with $H \to \infty$, $0 \le M \le CN$ and $H \le U\log N$. Let $L$ be the least even integer satisfying \eqref{eq:transfer-truncation-order}, and put $K := m + L + 1$.

\begin{hypothesis}[Logarithmic saving]
\label{hyp:logarithmic-saving}
There are fixed constants $A \ge 1$ and $\sigma > 0$ such that, uniformly for $2 \le k \le K$ and all $k$-element subsets $\mathcal{H}$ of $[H]$,
\begin{equation}
\label{eq:individual-hl-logarithmic-saving}
\left|\sum_{n \, = \, 1}^{N}\prod_{h \, \in \, \mathcal{H}}\mathbf{1}_{\mathcal{P}}(M + n + h) - \mathfrak{S}(\mathcal{H})\int_{\max\{2, M\}}^{M + N}\frac{dt}{(\log t)^{k}}\right| \ll \frac{A^{k}N}{(\log N)^{k + 1 + \sigma}}.
\end{equation}
The implied constant may depend on $m, C, U, \delta, A, \sigma$ but is independent of $k$, $\mathcal{H}$, $H$, $M$ and $N$.
\end{hypothesis}

\begin{hypothesis}[Power saving]
\label{hyp:power-saving}
There are fixed constants $A \ge 1$ and $0 < \theta < 1$ such that the absolute value on the left side of \eqref{eq:individual-hl-logarithmic-saving} is
\begin{equation}
\label{eq:individual-hl-power-saving}
O\left(A^{k}N^{1 - \theta}\right)
\end{equation}
uniformly over the same tuples and parameter ranges. The implied constant may depend on $m, C, U, \delta, A, \theta$ but is independent of $k$, $\mathcal{H}$, $H$, $M$ and $N$.
\end{hypothesis}

\begin{proposition}
\label{prop:individual-hl-sufficient-conditions}
With the parameters and ranges specified above, either Hypothesis~\ref{hyp:logarithmic-saving} or Hypothesis~\ref{hyp:power-saving} implies \eqref{eq:weighted-hl-sufficient-condition}. Hypothesis~\ref{hyp:logarithmic-saving} also implies the averaged condition \eqref{eq:normalized-averaged-hl-condition}.
\end{proposition}

\begin{proof}
If $E_{k}$ bounds the absolute residual for every $k$-element tuple, summing over the tuples and dividing by $N$ gives
\begin{equation}
\label{eq:individual-to-averaged-hl-error}
|\Delta_{\mathrm{HL}}(k)| \le \frac{1}{N}\binom{H}{k}E_{k} \le \frac{H^{k}}{Nk!}E_{k}.
\end{equation}
For $k > H$ the first bound is zero, so there is no exception in that range. Under \eqref{eq:individual-hl-logarithmic-saving}, we obtain
\begin{equation*}
|\Delta_{\mathrm{HL}}(k)| \ll \frac{1}{(\log N)^{1 + \sigma}}\frac{A^{k}}{k!}\left(\frac{H}{\log N}\right)^{k} \qquad (2 \le k \le K).
\end{equation*}
Since $H/(\log N)^{1 + \sigma} \le U/(\log N)^{\sigma} \to 0$, this proves \eqref{eq:normalized-averaged-hl-condition}. In fact, \eqref{eq:weighted-hl-factorial-envelope} and \eqref{eq:transfer-singleton-residual}, with $B > 2 + \sigma$, give
\begin{equation*}
\max_{j \, \in \, \{L, L + 1\}}|R_{\mathrm{HL}}(m, j)| \ll_{m, C, U, A, \sigma} \frac{1}{(\log N)^{1 + \sigma}}.
\end{equation*}
The fixed implied constant in the hypothesis is understood to be included in the implied constant here.

Under \eqref{eq:individual-hl-power-saving}, formula \eqref{eq:individual-to-averaged-hl-error} instead gives
\begin{equation*}
|\Delta_{\mathrm{HL}}(k)| \ll N^{-\theta}\frac{(AH)^{k}}{k!}.
\end{equation*}
This time we keep the sum finite. Taking $H$ large enough that $AH \ge 1$, Lemma~\ref{lem:weighted-hl-error-envelopes} gives
\begin{equation*}
\max_{j \, \in \, \{L, L + 1\}}|R_{\mathrm{HL}}(m, j)| \ll \frac{2}{N} + |\Delta_{\mathrm{HL}}(1)| + N^{-\theta}\frac{L + 2}{m!}(AH)^{K}.
\end{equation*}
Here we used $1/(k - m)! \le 1$ and bounded each of the at most $L + 2$ powers by $(AH)^{K}$. For fixed $m$, $A$ and $\delta$,
\begin{equation*}
\log\left((L + 2)(AH)^{K}\right) = O_{m, A, \delta}\left(\frac{(\log H)^{2}}{\log\log H}\right) = o(\log N).
\end{equation*}
The last equality follows from $H \le U\log N$. The last term in the preceding bound is therefore $N^{-\theta + o(1)}$, hence $o(1/H)$. The empty and singleton contributions are $o(1/H)$ by Lemma~\ref{lem:empty-and-singleton-hl-residuals}.
\end{proof}

The uniformity in $k$ in these hypotheses is essential. A statement with an unspecified implied constant depending on $k$, valid for each fixed $k$, does not justify substitution of $k \le K(H)$ tending to infinity. In contrast, \eqref{eq:individual-hl-power-saving} allows a stated exponential dependence $A^{k}$; the proof shows explicitly that its cost is $N^{o(1)}$ in our range. In particular, an error $O(N^{1/2 + \epsilon})$ uniform in these tuples, with fixed $0 < \epsilon < 1/2$, is more than sufficient.

The hypotheses were written for all sets of shifts to agree with the definition of $S(k; H)$. For inadmissible $\mathcal{H}$, the singular series is zero and the tuple count is at most $k$, as observed in Subsection~\ref{subsec:hardy-littlewood-residual}. Both proposed individual-tuple error bounds exceed $k$ uniformly for $2 \le k \le K$ once $N$ is sufficiently large: indeed, $K\log\log N = o(\log N)$, so the right side of \eqref{eq:individual-hl-logarithmic-saving} without its implied constant is $N^{1 - o(1)}$, while $N^{1 - \theta}$ also exceeds $K$. Thus the substantive assumptions concern admissible tuples of size at least two.

\subsection{The corrected Poisson asymptotic}
\label{subsec:conditional-corrected-poisson-asymptotic}

We finish by spelling out the conclusion when the interval length is comparable to the average spacing between primes. In this range the Poisson main term stays on a fixed scale for fixed $m$, while the arithmetic correction contains terms of orders $(\log H)/H$ and $1/H$. The remainder below is smaller than both of these scales.

\begin{theorem}
\label{thm:conditional-corrected-poisson-asymptotic}
Fix an integer $m \ge 0$, real numbers $C, \delta > 0$, and $0 < V \le U$. As $N \to \infty$, consider integers in the ranges
\begin{equation*}
0 \le M \le CN, \qquad V\log N \le H \le U\log N.
\end{equation*}
Let $L$ be the least even integer satisfying \eqref{eq:transfer-truncation-order}, and suppose that \eqref{eq:weighted-hl-sufficient-condition} holds uniformly in these ranges. This hypothesis follows, in particular, from \eqref{eq:normalized-averaged-hl-condition}, or from either individual-tuple hypothesis in Proposition~\ref{prop:individual-hl-sufficient-conditions}.

Put
\begin{equation*}
\lambda := \frac{H}{\log N + c(M/N)}, \qquad c(\alpha) := (\alpha + 1)\log(\alpha + 1) - \alpha\log\alpha - 1,
\end{equation*}
with $0\log 0 = 0$, as in Subsection~\ref{subsec:averaging-effective-mean}. Then, uniformly over the stated ranges,
\begin{equation}
\label{eq:conditional-corrected-poisson-asymptotic}
P(m; H, M, N) = \frac{e^{-\lambda}\lambda^{m}}{m!}\left\{1 - \frac{\log H + \log(2\pi) + \gamma - 1}{2H}\left((m - \lambda)^{2} - m\right)\right\} + o(1/H).
\end{equation}
The same formula holds with $\lambda$ replaced throughout by $\mu(H, M, N)$.
\end{theorem}

\begin{proof}
Corollary~\ref{cor:weighted-hl-transfer} gives \eqref{eq:conditional-mean-prime-count}. Equivalently, its integrated conclusion \eqref{eq:conditional-integrated-prime-count}, followed by Corollary~\ref{cor:averaging-shifted-parameter}, gives
\begin{equation*}
P(m; H, M, N) = Q(m; \lambda, H) + o(1/H).
\end{equation*}
Substitute the definitions of $Q$ and $\eta(H)$ to obtain \eqref{eq:conditional-corrected-poisson-asymptotic}. Substitution into \eqref{eq:conditional-mean-prime-count} gives the version with $\mu$.
\end{proof}

In particular, the parameter in \eqref{eq:conditional-corrected-poisson-asymptotic} is
\begin{equation*}
\lambda = \frac{H}{\log N - 1} \quad (M = 0), \qquad \lambda = \frac{H}{\log N + 2\log 2 - 1} \quad (M = N).
\end{equation*}
Each specialization requires the stated Hardy--Littlewood hypothesis on the corresponding averaging range; the case $M = N$ is included in the theorem when $C \ge 1$.

For a quantitative version, under \eqref{eq:individual-hl-logarithmic-saving} the error in \eqref{eq:conditional-corrected-poisson-asymptotic} is bounded by
\begin{equation*}
O_{m, C, U, V, A, \sigma, \delta, \epsilon}\left(H^{-3/2 + \epsilon} + H^{-1 - \delta + o(1)} + (\log N)^{-1 - \sigma} + (\log N)^{-2}\right)
\end{equation*}
for any fixed $0 < \epsilon < 1/2$, with the fixed implied constant in the hypothesis also allowed. This follows from \eqref{eq:transfer-integrated-bound}, the estimates in the proof of Corollary~\ref{cor:weighted-hl-transfer}, and Corollary~\ref{cor:averaging-shifted-parameter}. The factorial term can instead be retained in the explicit form \eqref{eq:transfer-analytic-error} if a bound without $o(1)$ in the exponent is preferred.

No hypothesis on the Montgomery--Soundararajan residual or on the truncation or density-averaging residuals is required. The sole unproved input is the specified control of prime-tuple counts. When $H/\log N \to \kappa > 0$, the theorem in particular recovers the Poisson limit $P(m; H, M, N) \to e^{-\kappa}\kappa^{m}/m!$; the displayed formula retains the arithmetic correction and the dependence of the effective mean on the averaging interval.

\begin{proof}[Proof of Theorem~\ref{thm:intro-corrected-poisson}]
The introductory Hardy--Littlewood hypothesis implies \eqref{eq:weighted-hl-sufficient-condition} by Proposition~\ref{prop:individual-hl-sufficient-conditions}: choose any fixed $0 < \epsilon < 1/2$, so that its error is a power saving in $N$, and note that the required tuple orders satisfy $m + L + 1 \le \log H$ for sufficiently large $H$. Apply Theorem~\ref{thm:conditional-corrected-poisson-asymptotic} with $M = 0$, observing that $c(0) = -1$ and hence $\lambda = H/(\log N - 1)$.
\end{proof}

\section{A random model for the arithmetic correction}
\label{sec:random-model-arithmetic-correction}

We now explain the arithmetic correction through a random model. Recall that
\begin{equation*}
Q(m; u, H) = \frac{e^{-u}u^{m}}{m!}\left\{1 - \frac{\eta(H)}{2}\left((m - u)^{2} - m\right)\right\}, \qquad \eta(H) = \frac{\log H + \log(2\pi) + \gamma - 1}{H},
\end{equation*}
with values at $u = 0$ understood by continuity. In the model below, local divisibility conditions reduce the count variance relative to the Poisson value $u$. The polynomial multiplying $\eta(H)$ records the first effect of this variance deficit on the probabilities of individual counts. We use the finite sieve from Section~\ref{sec:singular-series-finite-sieve}: its second moment suggests both the polynomial and its coefficient, while the higher-moment estimates in that section justify the approximation.

The construction belongs to an established family of refinements of Cram\'er's model. Granville and Lumley \cite[Sections~5 and~8]{GL2023} first sieve out multiples of small primes and then retain the surviving integers independently with an adjusted probability. Banks, Ford and Tao \cite[Section~1.3]{BFT2023} instead construct a random set by excluding independently chosen residue classes, with the sieve level depending on the location of the interval. We use random residue classes on a finite interval, followed by independent thinning. This gives a particularly direct explanation of the correction in the range where the expected count remains bounded.

\subsection{Sieve first, then toss coins}
\label{subsec:random-model-construction}

Fix $U > 0$, let $H$ tend to infinity through the integers, and prescribe a mean $0 \le u \le U$. As in Subsection~\ref{subsec:finite-sieve-probability-space}, put
\begin{equation*}
y := H^{3}, \qquad q := \prod_{p \, \le \, y}p, \qquad \rho := \prod_{p \, \le \, y}\left(1 - \frac{1}{p}\right), \qquad b := H\rho.
\end{equation*}
For each prime $p \le y$, choose a residue class $a_{p} \bmod p$ uniformly and independently, and remove all members of $[H]$ in that class. Let $Y$ be the number of survivors. Each position survives with probability $\rho$, so $\mathbb{E}Y = b$. The Chinese remainder theorem identifies these choices with a uniformly chosen translate modulo $q$, exactly as in the earlier construction.

Independently retain each survivor with probability
\begin{equation*}
\theta := \frac{u}{b},
\end{equation*}
and let $Z$ be the number retained. Mertens' product estimate gives $b \asymp H/\log H$, so there is a threshold $H_{0}(U)$ such that $0 \le \theta \le 1$ whenever $H \ge H_{0}(U)$ and $0 \le u \le U$. Throughout this section, statements about $Z$ are restricted to these ranges. Conditional on $Y$, the distribution of $Z$ is binomial with parameters $Y$ and $\theta$. Consequently,
\begin{equation}
\label{eq:random-model-mean-and-variance}
\mathbb{E}Z = u, \qquad \operatorname{Var}Z = \theta(1 - \theta)b + \theta^{2}\operatorname{Var}Y = u + \theta^{2}\left(\operatorname{Var}Y - b\right).
\end{equation}
The variance identity follows by adding the average conditional variance, $\theta(1 - \theta)b$, to the variance of the conditional mean, $\theta^{2}\operatorname{Var}Y$.

The final coin tosses are independent, but the set on which they act is random and reflects the residue classes removed by the sieve. Thus the indicators of retention at different positions are dependent after averaging over the sieve. If we retained every position independently with probability $u/H$, the result would instead have variance $u - u^{2}/H$. The difference between the two models can therefore be read from the variance of $Y$.

\subsection{The variance from pairs of survivors}
\label{subsec:random-model-pair-variance}

Consider two positions separated by $d$, where $1 \le d < H$. If $p \mid d$, they occupy the same residue class modulo $p$, and both survive this stage with probability $1 - 1/p$. Otherwise they occupy two classes, and the probability is $1 - 2/p$. Independence between the choices for different primes gives the joint survival probability
\begin{equation*}
a_{y}(d) := \prod_{\substack{p \, \le \, y \\ p \, \mid \, d}}\left(1 - \frac{1}{p}\right)\prod_{\substack{p \, \le \, y \\ p \, \nmid \, d}}\left(1 - \frac{2}{p}\right).
\end{equation*}
There are $H - d$ unordered pairs with separation $d$. Since $Y(Y - 1)$ counts ordered pairs of distinct survivors,
\begin{equation}
\label{eq:random-model-variance-pair-sum}
\operatorname{Var}Y - b = 2\sum_{d \, = \, 1}^{H - 1}(H - d)a_{y}(d) - H^{2}\rho^{2}.
\end{equation}
This is an exact calculation in the finite probability space. No conjecture about primes has entered it.

After division by $\rho^{2}$, the joint survival probability is the finite Euler product for the pair singular series. Write
\begin{equation*}
s(d) := \mathfrak{S}(\{0, d\}), \qquad \mathfrak{c} := \prod_{p \, > \, 2}\left(1 - \frac{1}{(p - 1)^{2}}\right).
\end{equation*}
The local factors give $s(d) = 0$ for odd $d$ and
\begin{equation}
\label{eq:random-model-even-pair-series}
s(2r) = 2\mathfrak{c}\prod_{\substack{p \, \mid \, r \\ p \, > \, 2}}\frac{p - 1}{p - 2} \qquad (r \ge 1).
\end{equation}
Thus the singular series arises here simply by comparing the probability that both positions survive with the product of their individual survival probabilities.

For $p > y > H$, the two positions are distinct modulo $p$, so
\begin{equation*}
s(d) = \frac{a_{y}(d)}{\rho^{2}}\prod_{p \, > \, y}\left(1 - \frac{1}{(p - 1)^{2}}\right), \qquad \prod_{p \, > \, y}\left(1 - \frac{1}{(p - 1)^{2}}\right) = 1 + O(1/y).
\end{equation*}
This includes odd $d$, when both sides of the first identity vanish. It remains to understand one weighted sum over the separations $d$.

\subsection{Where the logarithm comes from}
\label{subsec:random-model-logarithmic-deficit}

The main feature of that sum can be seen by the elementary argument in Goldston \cite[Lemma~3, pp.~161--162]{GOL1984}. We give the calculation because it explains the factor $\log H$ without requiring a general singular-series average.

Define $f(d)$ to be zero unless $d$ is odd and squarefree, and in the remaining cases put
\begin{equation*}
f(d) := \prod_{p \, \mid \, d}\frac{1}{p - 2}.
\end{equation*}
In particular, $f(1) = 1$. Expanding the product in \eqref{eq:random-model-even-pair-series} gives
\begin{equation*}
s(2r) = 2\mathfrak{c}\sum_{d \, \mid \, r}f(d).
\end{equation*}
Put $x = H/2$, allowing $x$ to be a half-integer. On changing the order of summation, we obtain
\begin{equation}
\label{eq:random-model-weighted-pair-divisors}
2\sum_{d \, = \, 1}^{H - 1}(H - d)s(d) = 8\mathfrak{c}\sum_{d \, \le \, x}d f(d)\sum_{r \, \le \, x/d}\left(\frac{x}{d} - r\right).
\end{equation}
The endpoint term is zero when $x/d$ is an integer. The useful observation is the elementary identity
\begin{equation*}
\sum_{r \, \le \, v}(v - r) = \frac{v^{2}}{2} - \frac{v}{2} + O(1) \qquad (v \ge 1).
\end{equation*}
The quadratic term produces the leading term $H^{2}$. The linear term produces the logarithmic deficit. More explicitly, \eqref{eq:random-model-weighted-pair-divisors} becomes
\begin{equation}
\label{eq:random-model-pair-sum-three-terms}
2\sum_{d \, = \, 1}^{H - 1}(H - d)s(d) = 4\mathfrak{c}x^{2}\sum_{d \, \le \, x}\frac{f(d)}{d} - 4\mathfrak{c}x\sum_{d \, \le \, x}f(d) + O\left(\sum_{d \, \le \, x}d f(d)\right).
\end{equation}
The estimates needed here are
\begin{equation}
\label{eq:random-model-divisor-weight-estimates}
\sum_{d \, \le \, x}\frac{f(d)}{d} = \frac{1}{\mathfrak{c}} + O(1/x), \qquad \sum_{d \, \le \, x}f(d) = \frac{\log x}{2\mathfrak{c}} + O(1), \qquad \sum_{d \, \le \, x}d f(d) \ll x.
\end{equation}

For completeness, these estimates have a short verification. Define a multiplicative function $g$ by $g(1) = 1$, $g(2) = -1/2$, and, for odd primes $p$, by 
\begin{equation*}
g(p) = \frac{2}{p(p - 2)} \qquad \text{and} \qquad g(p^{2}) = -\frac{1}{p(p - 2)}.
\end{equation*}
Its values at all other prime powers are zero. Checking prime powers gives
\begin{equation*}
f(n) = \sum_{d r \, = \, n}\frac{g(d)}{r}, \qquad \sum_{d \, = \, 1}^{\infty}|g(d)|(1 + \log d) < \infty, \qquad \sum_{d \, = \, 1}^{\infty}g(d) = \frac{1}{2\mathfrak{c}}.
\end{equation*}
The absolute convergence follows from $|g(p)| + |g(p^{2})| \ll p^{-2}$ and the convergence of the same prime sum with an additional factor $\log p$. The last identity follows by multiplying the local factors: the factor at $2$ is $1/2$, and the factor at an odd prime is $1 + 1/(p(p - 2))$. Summing the convolution and using the harmonic sum now gives
\begin{equation*}
\sum_{n \, \le \, x}f(n) = \sum_{d \, \le \, x}g(d)\left(\log(x/d) + O(1)\right) = \frac{\log x}{2\mathfrak{c}} + O(1).
\end{equation*}
Also, the Euler product gives $\sum_{d \ge 1}f(d)/d = 1/\mathfrak{c}$. Partial summation applied to the estimate just proved bounds its tail by $O(1/x)$ and gives $\sum_{d \le x}d f(d) \ll x$. This proves \eqref{eq:random-model-divisor-weight-estimates}.

Substituting these estimates into \eqref{eq:random-model-pair-sum-three-terms}, and recalling $x = H/2$, yields
\begin{equation}
\label{eq:random-model-leading-pair-deficit}
2\sum_{d \, = \, 1}^{H - 1}(H - d)s(d) = H^{2} - H\log H + O(H).
\end{equation}
The logarithm has come from the harmonic growth of the divisor weights, multiplied by the linear term in the triangular sum. In particular, its appearance does not require a conjecture about prime tuples or an estimate involving large tuple sizes.

Keeping the term of order $H$ requires a more precise evaluation. Goldston \cite[equation~(33), p.~295]{GOL1990} records
\begin{equation}
\label{eq:random-model-full-pair-average}
2\sum_{d \, = \, 1}^{H - 1}(H - d)s(d) = H^{2} - H\left(\log H + \log(2\pi) + \gamma - 1\right) + O_{\epsilon}(H^{1/2 + \epsilon}).
\end{equation}
This is also the pair average used by Montgomery and Soundararajan \cite[equations~(47)--(48)]{MS2004}. The elementary calculation above identifies the leading deficit; \eqref{eq:random-model-full-pair-average} supplies its constant term.

Replacing $s(d)$ by $a_{y}(d)/\rho^{2}$ in this sum costs $O(H^{2}/y)$, by the common Euler-product tail and \eqref{eq:random-model-leading-pair-deficit}. Equations \eqref{eq:random-model-mean-and-variance} and \eqref{eq:random-model-variance-pair-sum} therefore give
\begin{equation}
\label{eq:random-model-variance-deficit}
\operatorname{Var}Z = u - u^{2}\eta(H) + O_{U, \epsilon}(H^{-3/2 + \epsilon}), \qquad \eta(H) = \frac{\log H + \log(2\pi) + \gamma - 1}{H}.
\end{equation}
Here $y = H^{3}$ makes the error from the omitted Euler factors smaller than the displayed error. The coefficient of $u^{2}$ in the variance deficit is thus larger by a factor of order $\log H$ than the coefficient $1/H$ in the independent binomial model.

\subsection{From the variance deficit to the correction}
\label{subsec:random-model-charlier-correction}

We next explain why a variance deficit suggests the particular polynomial in $Q$. For a bounded nonnegative integer-valued random variable $W$ of mean $u$, expansion at $w = 0$ gives
\begin{equation}
\label{eq:random-model-factorial-cumulants}
\log\mathbb{E}(1 + w)^{W} = uw + \frac{\operatorname{Var}W - u}{2}w^{2} + O_{W}(w^{3}).
\end{equation}
Indeed, the first two coefficients of $\mathbb{E}(1 + w)^{W}$ are $\mathbb{E}W$ and $\mathbb{E}(W(W - 1))/2$, and taking the logarithm subtracts $u^{2}/2$ from the second coefficient. These coefficients of the logarithm, multiplied by the corresponding factorials, are called the \emph{factorial cumulants}. Only the first two are needed for this explanation.

If the higher terms are negligible at the required precision, substituting \eqref{eq:random-model-variance-deficit} into \eqref{eq:random-model-factorial-cumulants} and exponentiating suggests
\begin{equation*}
\mathbb{E}z^{Z} \approx e^{u(z - 1)}\left(1 - \frac{\eta(H)u^{2}}{2}(z - 1)^{2}\right).
\end{equation*}
This step is a prediction from the second moment, not a consequence of the second moment alone. In particular, the Taylor remainder in \eqref{eq:random-model-factorial-cumulants} need not be uniform as the distribution changes with $H$. To obtain probabilities of individual counts, one needs control of the generating function beyond a merely formal expansion at $z = 1$.

For the present model that control was proved in Proposition~\ref{prop:singular-series-generating-function}. Uniformly for $0 \le u \le U$ and complex $z$ with $|z| \le 2$,
\begin{equation}
\label{eq:random-model-generating-function}
\mathbb{E}z^{Z} = e^{u(z - 1)}\left(1 - \frac{\eta(H)u^{2}}{2}(z - 1)^{2}\right) + O_{U, \epsilon}(H^{-3/2 + \epsilon}).
\end{equation}
The mechanism is visible without repeating that proof. Conditional on $Y$, the generating function is $(1 + \theta(z - 1))^{Y}$. Writing $Y = b + (Y - b)$, the binomial expansion contributes $-b\theta^{2}(z - 1)^{2}/2$, while averaging the centered fluctuation contributes $\theta^{2}\operatorname{Var}Y\,(z - 1)^{2}/2$. Their sum is exactly the variance-minus-mean term in \eqref{eq:random-model-mean-and-variance}. A uniform upper bound for $Y$ and its fourth centered moment control the remaining terms. This is where the earlier proof makes the second-moment prediction rigorous.

Put $p_{m}(u) := e^{-u}u^{m}/m!$ for $m \ge 0$. Since $e^{u(z - 1)}$ is its generating function and differentiation twice with respect to $u$ introduces $(z - 1)^{2}$, the coefficient of $z^{m}$ in the main expression in \eqref{eq:random-model-generating-function} is
\begin{equation*}
p_{m}(u) - \frac{\eta(H)u^{2}}{2}p_{m}''(u) = \frac{e^{-u}u^{m}}{m!}\left\{1 - \frac{\eta(H)}{2}\left((m - u)^{2} - m\right)\right\} = Q(m; u, H).
\end{equation*}
Values at $u = 0$ are understood by continuity. The polynomial $((m - u)^{2} - m)$ is the second Poisson--Charlier polynomial in the normalization used here. Cauchy's coefficient formula on $|z| = 2$ now gives the unconditional model estimate
\begin{equation}
\label{eq:random-model-count-distribution}
\mathbb{P}(Z = m) = Q(m; u, H) + O_{U, \epsilon}(2^{-m}H^{-3/2 + \epsilon}) \qquad (m \ge 0).
\end{equation}
The estimate holds uniformly for $H \ge H_{0}(U)$, $0 \le u \le U$ and integers $m \ge 0$. The implied constant depends only on $U$ and $\epsilon$. For $0 < \epsilon < 1/2$, the error is $o(1/H)$.

Thus the local residue conditions predict the variance deficit, and that deficit determines the first correction to the Poisson probabilities once the remaining terms are controlled. Independent Bernoulli trials give the same polynomial, with $1/H$ in place of $\eta(H)$. The polynomial reflects the form of the second-moment correction; its coefficient records the arithmetic.

\subsection{What the model explains}
\label{subsec:random-model-interpretation}

The correction preserves both the total mass and the mean. Its generating function is
\begin{equation*}
\sum_{m \, = \, 0}^{\infty}Q(m; u, H)z^{m} = e^{u(z - 1)}\left(1 - \frac{\eta(H)u^{2}}{2}(z - 1)^{2}\right).
\end{equation*}
Evaluating at $z = 1$ and differentiating there gives the exact identities
\begin{equation*}
\sum_{m \, = \, 0}^{\infty}Q(m; u, H) = 1, \qquad \sum_{m \, = \, 0}^{\infty}m Q(m; u, H) = u
\end{equation*}
and
\begin{equation*}
\sum_{m \, = \, 0}^{\infty}(m - u)^{2}Q(m; u, H) = u - u^{2}\eta(H).
\end{equation*}
All these sums converge absolutely. For $u > 0$, however, $Q(m; u, H)$ becomes negative when $m$ is sufficiently large. It is therefore a signed approximation to a probability mass function. The variable $Z$ supplies a genuine probability distribution approximated by it in \eqref{eq:random-model-count-distribution}.

Relative to the Poisson probabilities, the correction is positive when $(m - u)^{2} < m$ and negative when $(m - u)^{2} > m$. This describes precisely the bias toward central counts. It also makes clear why adjusting only the mean cannot account for the effect: the correction leaves the mean unchanged and reduces the second centered moment.

Granville and Lumley \cite[Appendix~B]{GL2023} observe that their numerical distributions, grouped by the number of survivors of a small-prime sieve, are narrower than the corresponding binomial predictions. They compare the variances with the prediction of Montgomery and Soundararajan and suggest incorporating further dependence into the model. The calculation above describes such dependence and its effect on individual count probabilities when the mean is bounded. Their questions about extreme counts and tails involve other ranges; the absolute error in \eqref{eq:random-model-count-distribution} does not give a relative approximation to very small tail probabilities.

There is also a methodological connection with Banks, Ford and Tao \cite[Section~2.6]{BFT2023}: they explain how interpreting singular-series sums probabilistically can avoid restrictions in the available uniformity in the number of shifts. Here the same viewpoint permits the full generating function to be handled through the survivor count, with precise asymptotics needed only for low moments.

Finally, the parameter $u$ is local. To model intervals starting near $t$, we take $u = H/\log t$ and then average over the starting points. For $H \le U\log N$, this parameter is bounded on $t \ge \sqrt{N}$; the initial portion is negligible at our precision, as in the density-averaging argument. Averaging the local approximation leads to $F(m; H, M, N)$, and Section~\ref{sec:density-averaging} replaces this by $Q$ at the averaged density $\mu$, or at the shifted parameter $\lambda$. In particular, the shift in the denominator of $\lambda$ comes from averaging the varying density. The factor $\eta(H)$ comes from the pair correlations within an interval. These are separate effects.

All statements about the random variables above are unconditional. Applying the model to primes is a heuristic step: a uniform random translate modulo the very large modulus $q$ is not the same as an integer starting point chosen from $(M, M + N]$. The Hardy--Littlewood hypotheses in Section~\ref{sec:prime-tuples-to-prime-counts} give the conditional passage to the actual prime-count distribution. The model explains how one can anticipate the correction from local divisibility and a second-moment calculation before carrying out that passage.

\subsection{Further corrections}
\label{subsec:further-corrections}

Further corrections would require more precise information about higher centered singular-series sums. Kuperberg \cite[Theorem~1.2]{KUP2025} proves that the sum over ordered triples of distinct shifts is $O(H(\log H)^{5})$ and conjectures the smaller scale $H(\log H)^{2}$ \cite[Conjecture~1.1]{KUP2025}. Such estimates are relevant to the third factorial cumulant of the sieve model and hence to further Poisson--Charlier terms. They do not yet identify the next correction: this would also require a sharper treatment of the pair-average remainder and control of the remaining higher-order contributions.

\clearpage

\appendix

\section{Numerical comparisons}
\label{sec:numerical-comparisons}

We compare the empirical distribution with the corrected approximations and with the Poisson and binomial predictions of the independent model. Throughout this section $M = 0$, and we suppress $H$ and $N$ where convenient. Thus
\begin{equation*}
P(m) := P(m; H, 0, N) = \frac{1}{N}\left|\{1 \le n \le N : \pi(n + H) - \pi(n) = m\}\right|.
\end{equation*}
We use nine pairs $(N, H)$, chosen by
\begin{equation*}
N \in \{10^{7}, 10^{8}, 10^{9}\}, \qquad \kappa \in \{1, 3, 6\}, \qquad H = \left\lfloor\kappa(\log N - 1) + \frac{1}{2}\right\rfloor.
\end{equation*}
Here $\kappa$ is a target mean; all predictions use the actual integer $H$. These choices keep the effective mean close to $\kappa$ as $N$ increases. A sieve generates the primes needed through $N + H$, and a moving-window counter tallies the starting points giving each prime count. All $N$ intervals are counted, including those whose endpoints are prime. The implementation, exact counts, supplementary comparisons and animations are available with the accompanying code and Quarto book \cite{FRE2026}.

Write $p_{m}(u) := e^{-u}u^{m}/m!$. The averaged Poisson prediction and its arithmetic correction are
\begin{equation*}
F_{0}(m) := \frac{1}{N}\int_{2}^{N}p_{m}\left(\frac{H}{\log t}\right)dt
\end{equation*}
and
\begin{equation*}
F(m) := F(m; H, 0, N) = \frac{1}{N}\int_{2}^{N}Q\left(m; \frac{H}{\log t}, H\right)dt,
\end{equation*}
where $Q$ is defined in \eqref{eq:corrected-poisson-expression}. To distinguish the arithmetic correction from the ordinary binomial correction, we also use
\begin{equation*}
B_{\mathrm{av}}(m) := \frac{1}{N}\int_{e}^{N}\binom{H}{m}\left(\frac{1}{\log t}\right)^{m}\left(1 - \frac{1}{\log t}\right)^{H - m}dt \qquad (0 \le m \le H),
\end{equation*}
with $B_{\mathrm{av}}(m) = 0$ for $m > H$. This averages the local binomial approximation to Cram\'er's model; the lower limit $e$ ensures that the success probability is at most $1$. Finally, put
\begin{equation*}
\lambda := \frac{H}{\log N - 1}, \qquad \mu := \frac{H}{N}\int_{2}^{N}\frac{dt}{\log t},
\end{equation*}
and write $Q_{\lambda}(m) := Q(m; \lambda, H)$ and $Q_{\mu}(m) := Q(m; \mu, H)$. The former is the expression in Theorem~\ref{thm:intro-corrected-poisson}; the latter retains the exact averaged density. None of these parameters is fitted to the observed frequencies.

For an approximation $A$, we measure the discrepancies by
\begin{equation*}
D_{1}(A) := \sum_{m \, = \, 0}^{\infty}|P(m) - A(m)|, \qquad D_{2}(A) := \left(\sum_{m \, = \, 0}^{\infty}|P(m) - A(m)|^{2}\right)^{1/2}.
\end{equation*}
\begin{table}[!tbp]
\centering
\caption{Discrepancies for the nine pairs $(N, H)$, rounded to five decimal places. The three lengths at each scale target means $1$, $3$ and $6$, respectively. Smaller values indicate closer agreement with the empirical proportions.}
\label{tab:numerical-discrepancies}
\small
\setlength{\tabcolsep}{7pt}
\renewcommand{\arraystretch}{1.08}
\begin{tabular}{@{}rrrrrr@{}}
\toprule
\multicolumn{6}{c}{$D_{1}$ discrepancy} \\
\midrule
$N$ & $H$ & $F_{0}$ & $B_{\mathrm{av}}$ & $F$ & $Q_{\lambda}$ \\
\midrule
$10^{7}$ & 15 & 0.17400 & 0.13604 & 0.02914 & 0.02772 \\
 & 45 & 0.19931 & 0.16732 & 0.04997 & 0.04532 \\
 & 91 & 0.23342 & 0.20190 & 0.06765 & 0.05322 \\
\addlinespace[3pt]
$10^{8}$ & 17 & 0.15485 & 0.12208 & 0.02460 & 0.02377 \\
 & 52 & 0.17661 & 0.14736 & 0.03929 & 0.03607 \\
 & 105 & 0.20266 & 0.17519 & 0.05221 & 0.04085 \\
\addlinespace[3pt]
$10^{9}$ & 20 & 0.14254 & 0.11372 & 0.02277 & 0.02202 \\
 & 59 & 0.15905 & 0.13323 & 0.03242 & 0.03009 \\
 & 118 & 0.17975 & 0.15544 & 0.04232 & 0.03389 \\
\addlinespace[7pt]
\midrule
\multicolumn{6}{c}{$D_{2}$ discrepancy} \\
\midrule
$N$ & $H$ & $F_{0}$ & $B_{\mathrm{av}}$ & $F$ & $Q_{\lambda}$ \\
\midrule
$10^{7}$ & 15 & 0.08882 & 0.06965 & 0.01479 & 0.01358 \\
 & 45 & 0.07931 & 0.06683 & 0.01986 & 0.01684 \\
 & 91 & 0.07512 & 0.06526 & 0.02188 & 0.01656 \\
\addlinespace[3pt]
$10^{8}$ & 17 & 0.07937 & 0.06272 & 0.01234 & 0.01159 \\
 & 52 & 0.06949 & 0.05864 & 0.01566 & 0.01352 \\
 & 105 & 0.06486 & 0.05623 & 0.01667 & 0.01290 \\
\addlinespace[3pt]
$10^{9}$ & 20 & 0.07195 & 0.05741 & 0.01069 & 0.01015 \\
 & 59 & 0.06217 & 0.05259 & 0.01304 & 0.01145 \\
 & 118 & 0.05771 & 0.05003 & 0.01356 & 0.01073 \\
\bottomrule
\end{tabular}
\end{table}

Both sums include prime counts that never occur in the data. The predictions are neither clipped nor renormalized: $F$ and $F_{0}$ have total mass $1 - 2/N$, while $B_{\mathrm{av}}$ has total mass $1 - e/N$. The corrected expressions may take negative values, which enter the discrepancies with their signs intact. In these experiments the total negative mass $\sum_{m \, \ge \, 0}\max\{-F(m), 0\}$ ranges from $0.00145$ to $0.00367$.

The integrals are evaluated by adaptive numerical quadrature after the substitution $t = a e^{s}$, where $a$ is the lower integration limit. The sums defining the discrepancies are truncated beyond both $H$ and the largest observed count, with omitted absolute tails less than $10^{-12}$ for every displayed approximation. Quadrature tolerances were tightened and the predictions recalculated separately; the resulting discrepancies agreed well beyond the five decimal places reported below. A further comparison with $30$-digit quadrature for selected bins of the representative case gave differences below $4 \times 10^{-16}$ for $F$ and $F_{0}$. These are numerical checks, rather than certified bounds for the quadrature error.

Table~\ref{tab:numerical-discrepancies} shows that $F$ improves on both $F_{0}$ and $B_{\mathrm{av}}$ in every case and in both norms. Relative to $F_{0}$, the reduction is approximately $71$--$84\%$ in $D_{1}$ and $71$--$85\%$ in $D_{2}$. Both $Q_{\lambda}$ and $Q_{\mu}$ give smaller discrepancies than $F$ in all nine cases; the values for $Q_{\mu}$ are included in the supplementary data. The integral arising naturally in the proof need not be the most accurate approximation at a particular finite scale.

Figure~\ref{fig:numerical-prime-count-distribution} shows the preselected representative case $N = 10^{9}$, $H = 59$. The lower panel compares the observed departure from the averaged Poisson prediction with the predicted arithmetic correction. The correction captures the sign and much of the size of the departure, though a visible discrepancy remains near the center. These computations support the predicted correction across the nine prescribed cases; they do not isolate its constant term or establish an asymptotic error rate. The summed discrepancies also assess the whole distribution, whereas the conditional theorem is stated for each fixed $m$.

\begin{figure}[!htbp]
\centering
\includegraphics[width=\textwidth]{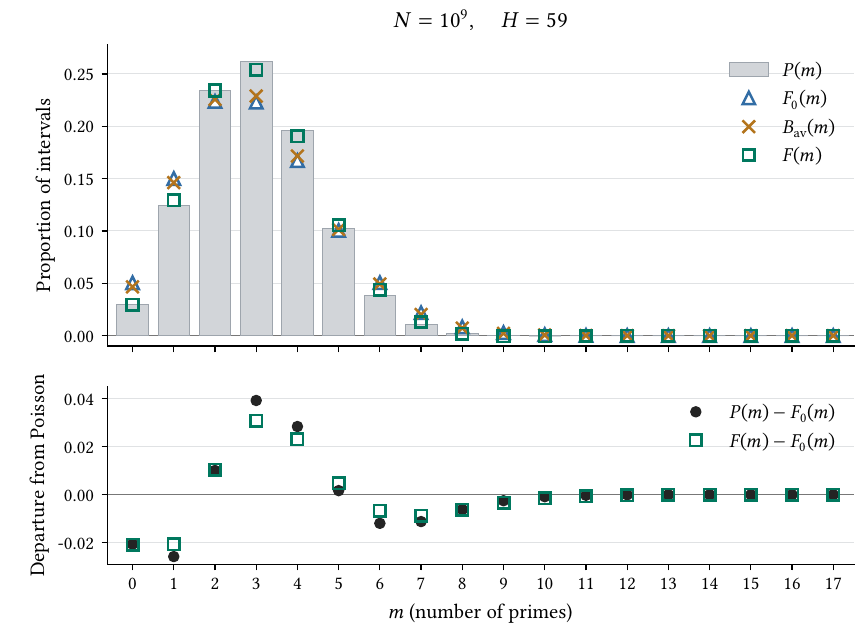}
\caption{Prime counts for $N = 10^{9}$ and $H = 59$, with $\lambda = 2.99139\ldots$ and $\mu = 3.00010\ldots$. Above: empirical proportions $P(m)$ and the predictions $F_{0}(m)$, $B_{\mathrm{av}}(m)$ and $F(m)$. Below: the observed departure $P(m) - F_{0}(m)$ and the predicted correction $F(m) - F_{0}(m)$. Only integer values of $m$ are plotted. The displayed range contains all observed counts; the discrepancy sums also include the predicted tails. Negative values of $F$ are retained, with minimum approximately $-0.000654$ at $m = 10$.}
\label{fig:numerical-prime-count-distribution}
\end{figure}

\clearpage

\section{From Gauss to Cram\'er}
\label{app:gauss-to-cramer}

The distribution of prime counts in short intervals has a rather distinguished numerical precedent. Among the tables preserved in Gauss's \emph{Nachlass} are tallies of the number of blocks of one hundred consecutive integers containing exactly $m$ primes. They record the same kind of statistic as the empirical distributions in Appendix~\ref{sec:numerical-comparisons}, with the starting points restricted to multiples of one hundred.

\subsection{Gauss's counts}
\label{subsec:gauss-centades}

In his letter to Johann Franz Encke of December 24, 1849, Gauss recalled counting primes in \emph{chiliads}, or blocks of one thousand integers, as early as 1792 or 1793 \cite[pp.~444--447]{GAU1863}. These counts suggested that the frequency of primes near $t$ was approximately $1/\log t$. Integrating this proposed density gives
\begin{equation*}
\pi(N) \approx \int_{2}^{N}\frac{dt}{\log t}.
\end{equation*}
The density was the starting point; the estimate for the cumulative count followed by integration. In the same letter, while discussing Legendre's approximation $N/(\log N - 1.08366)$ (see \cite[p.~65, \S VIII, no.~394]{LEG1830}), Gauss suggested that the differential of the counting function should be simpler than the function itself \cite[p.~446]{GAU1863}.

Gauss described returning to the tables in an ``idle quarter of an hour''; an English translation of this passage is reproduced by Tschinkel \cite{TSC2006}. Goldschmidt later filled gaps in the first million and extended the count through the first three million, using published factor tables. For the second and third million, Gauss prescribed a scheme that retained more information than the total number of primes. It recorded how many blocks of one hundred integers contained each possible prime count. These blocks are called \emph{centades} in the manuscript, and \emph{Hecatontaden} in the letter's explanation of the scheme \cite[p.~447]{GAU1863}.

Figure~\ref{fig:gauss-centades} shows the summary for the third million. The ten columns headed $210,220,\ldots,300$ have right endpoints in units of $10^{4}$: the first covers $2{,}000{,}000$ to $2{,}100{,}000$, and the last $2{,}900{,}000$ to $3{,}000{,}000$. Each column accounts for $1{,}000$ centades. The row labeled $m$ records the number containing exactly $m$ primes, and the final column aggregates the $10{,}000$ centades in the whole million. Multiplying each row total by $m$ and summing recovers the reported prime count, $67{,}862$. The correct count is $67{,}883$, a difference of just $21$; the logarithmic integral written beneath the table is approximately $67{,}915.73$. Thus the actual mean is $6.7883$ primes per centade. The rightmost column records the whole distribution around that mean.

\begin{figure}[!htbp]
\centering
\includegraphics[width=0.88\textwidth]{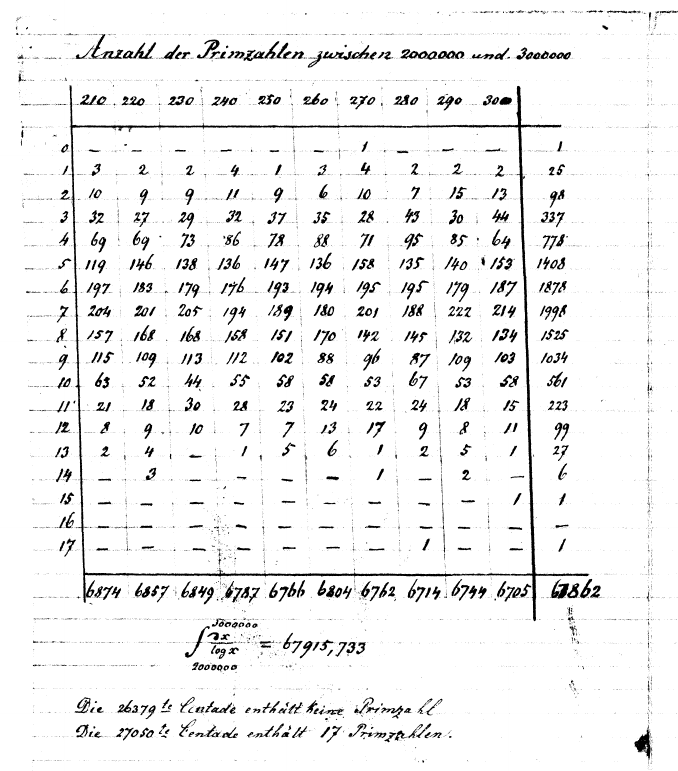}
\caption{The centade tally for primes between two and three million, from Gauss's \emph{Nachlass}; a printed version appears in \cite[p.~443]{GAU1863}. The entries reproduce the historical computation, including its errors. Source: Nieders\"achsische Staats- und Universit\"atsbibliothek G\"ottingen, mathematical \emph{Nachlass}, Math.~18; page~20 of the manuscript scan made available by Yuri Tschinkel and cited in \cite{TSC2006}.}
\label{fig:gauss-centades}
\end{figure}

These are non-overlapping intervals of fixed length $100$. Our asymptotic problem averages over every integer starting point and lets $H$ grow on the scale of the logarithm. The connection is the question being asked: how often does an interval contain a prescribed number of primes? The table supplies an empirical distribution, but the letter does not propose a Poisson law for it.

Riemann's memoir of 1859 \cite{RIE1859} connected the prime count with the zeros of the zeta function. Hadamard \cite{HAD1896} and de la Vall\'ee Poussin \cite{DEL1896} proved the prime number theorem in 1896, establishing the first-order asymptotic suggested by the logarithmic integral. A prediction for the full distribution of short-interval counts requires a further idea.

\subsection{From a density to coin flips}
\label{subsec:historical-cramer-model}

The indicator notation makes that idea particularly simple. The prime count in an interval is a sum of zeros and ones:
\begin{equation*}
X(n; H) = \sum_{h \, = \, 1}^{H}\mathbf{1}_{\mathcal{P}}(n + h).
\end{equation*}
Cram\'er's model \cite{CRA1935, CRA1936} replaces the indicators $\mathbf{1}_{\mathcal{P}}(j)$, for $j \ge 3$, by independent Bernoulli variables $B_{j}$ with
\begin{equation*}
\mathbb{P}(B_{j} = 1) = \frac{1}{\log j}, \qquad \mathbb{P}(B_{j} = 0) = 1 - \frac{1}{\log j}.
\end{equation*}
In other words, an independent coin flip decides whether each integer is declared prime, with a success probability chosen to match the proposed density.

For a short interval near a large integer $n$, freeze the probabilities at $1/\log n$. This gives a binomial approximation to the original model. Write $Y$ for the sum of these $H$ independent trials, and put $u := H/\log n$. Choosing the $m$ successful trials gives
\begin{equation*}
\mathbb{P}(Y = m) = \binom{H}{m}\left(\frac{u}{H}\right)^{m}\left(1 - \frac{u}{H}\right)^{H - m}.
\end{equation*}
For fixed $m$, the passage to a Poisson probability needs only a finite product and the elementary exponential limit. Indeed,
\begin{equation*}
\mathbb{P}(Y = m) = \frac{u^{m}}{m!}\left[\prod_{i \, = \, 0}^{m - 1}\left(1 - \frac{i}{H}\right)\right]\left(1 - \frac{u}{H}\right)^{H}\left(1 - \frac{u}{H}\right)^{-m}.
\end{equation*}
If $H \to \infty$ and $u \to \kappa > 0$, the finite product and the last factor tend to $1$, while the middle factor tends to $e^{-\kappa}$. Thus
\begin{equation*}
\mathbb{P}(Y = m) \to \frac{e^{-\kappa}\kappa^{m}}{m!}.
\end{equation*}
This permits $H$ to be an integer throughout: for example, take $H = \lfloor\kappa\log n\rfloor$. Freezing the density suffices for this first-order heuristic. At the precision of our secondary term, the varying density must be retained and averaged as in Section~\ref{sec:density-averaging}.

\subsection{Arithmetic correlations on average}
\label{subsec:historical-arithmetic-correlations}

Independence is the vulnerable step. For $a > 2$, the indicators $\mathbf{1}_{\mathcal{P}}(a)$ and $\mathbf{1}_{\mathcal{P}}(a + 1)$ cannot both be $1$: one of the two integers is even. More generally, divisibility conditions create correlations between shifted prime indicators.

To give the probabilistic language a precise meaning for the actual primes, choose $a$ uniformly among the integers in $(n,2n]$. For a fixed set $\mathcal{H}$ of $k$ distinct nonnegative shifts, the Hardy--Littlewood conjecture predicts, to first order,
\begin{equation*}
\mathbb{E}\prod_{h \, \in \, \mathcal{H}}\mathbf{1}_{\mathcal{P}}(a + h) = \frac{\mathfrak{S}(\mathcal{H}) + o(1)}{(\log n)^{k}} \qquad (n \to \infty).
\end{equation*}
The expectation here is over the starting point $a$. The independent model has the same first-order expression with $\mathfrak{S}(\mathcal{H})$ replaced by $1$. The singular series records the arithmetic dependence; for example, $\mathfrak{S}(\{0,1\}) = 0$.

The reason the independent model nevertheless gives the expected first-order answer is Gallagher's singular-series average \cite{GAL1976}: for each fixed $k$,
\begin{equation*}
\frac{1}{\binom{H}{k}}\sum_{\substack{\mathcal{H} \, \subseteq \, [H] \\ |\mathcal{H}| \, = \, k}}\mathfrak{S}(\mathcal{H}) \to 1 \qquad (H \to \infty).
\end{equation*}
The singular series can be far from $1$ for an individual set of shifts, but its average over those sets tends to $1$. These are precisely the averages that enter inclusion-exclusion for the event $X(a; H) = m$. With sufficient uniformity in the Hardy--Littlewood conjecture, their leading terms therefore give the same Poisson limit as independent coin flips.

The arithmetic dependence has not disappeared. The more precise average of Montgomery and Soundararajan \cite[equation~(17)]{MS2004} retains a contribution that the first-order limit discards. Our argument carries that contribution through inclusion-exclusion, producing the correction of order $(\log H)/H$. In this sense the singular-series average explains both why Cram\'er's prediction succeeds to first order and why arithmetic reappears in the next term.

\clearpage

\section{AI assistance, provenance and verification}
\label{app:ai-provenance}

The question studied in this paper predates the use of generative AI in its preparation. In the concluding remarks of \cite[Section~5]{FRE2018}, the author proposed obtaining lower-order terms in the distribution of prime counts in short intervals by combining inclusion--exclusion with more precise singular-series averages and a sufficiently uniform Hardy--Littlewood hypothesis. The project subsequently appeared in his research statements. Related code, numerical experiments and exposition have been available in his GitHub repository since 2023; version \texttt{1.1.0} of that repository is archived in the 2026 release \cite{FRE2026}.

Before the present collaboration, the author had developed the formal approximation and the inclusion--exclusion argument and investigated the predicted bias numerically. He had also identified the necessary scale $(\log H)/\log\log H$ for the truncation index. A calculation communicated by Andrew Granville in June 2023 clarified the treatment of the logarithmic integral and corrected the coefficient of the secondary term. The unresolved difficulty was to control the singular-series contribution when the truncation index grows with $H$. This was the main obstacle to completing the project.

The author used GPT-6~Astra (OpenAI) in an extended dialogue to resume the project, supplying the earlier work and relevant mathematical sources. The principal new proof idea supplied by GPT-6~Astra was the finite-sieve argument in Section~\ref{sec:singular-series-finite-sieve}. It proposed representing the singular-series sums through factorial moments of a finite random sieve, combining a deterministic bound for the number of survivors with low-moment estimates, and using a generating function to control the required alternating sums. This resolved the obstacle without an additional conjecture about singular-series averages or an extension of a fixed-order asymptotic to growing tuple sizes.

GPT-6~Astra also contributed to the organization of the residuals, the development and drafting of proofs, the density-averaging argument in Section~\ref{sec:density-averaging}, the interpretation through a random model, and the exposition. The author directed the investigation, questioned proposed arguments, checked the proofs and revised the text through repeated exchanges. Claude Fable~5.1 (Anthropic) assisted with revising and checking the numerical code and with preparing the computations and plots reported in Appendix~\ref{sec:numerical-comparisons}. The original research program and the earlier heuristic, inclusion--exclusion and numerical work were the author's; the substantial mathematical and computational contributions of the AI systems are identified above.

Drafts of the mathematical arguments were submitted to Claude Fable~5.1 for separate audits, with particular attention to the finite-sieve proof and the random-model section. The author examined the reports and checked the issues they raised; these audits informed further revisions. They supplemented the author's own verification of the proofs.

The author has checked the arguments, including some determined attempts to break them, and is prepared to defend the proofs presented here. He takes responsibility for the correctness of the mathematical claims, the use of sources and the final text. Any error found will require correction or, if necessary, withdrawal of the affected claims.

\end{document}